\documentclass[12pt]{amsart}
\usepackage{xcolor}
\usepackage{ulem}
\usepackage{dirtytalk}
\usepackage{amsmath,amssymb,setspace,nicefrac, yhmath, amscd,eucal}
\usepackage[active]{srcltx}
\usepackage[colorlinks, linkcolor=red, citecolor=blue, urlcolor=blue, hypertexnames=true]{hyperref}
\usepackage{amsrefs}
\usepackage{tikz-cd}
\allowdisplaybreaks
\usepackage[all]{xy}

\newcommand{\C}{{\mathbb C}}
\newcommand{\R}{{\mathbb R}}
\newcommand{\Q}{{\mathbb Q}}

\newcommand{\N}{{\mathbb N}}

\newcommand{\cC}{{\mathcal C}}

\newcommand{\cM}{{\mathcal M}}
\newcommand{\cP}{{\mathcal P}}
\newcommand{\cN}{{\mathcal N}}

\newcommand{\Aut}{{\rm Aut}}

\newcommand{\RN}{{\rm RN}}
\def\d{\mathrm{d}}
\def\Prob{\mathrm{Prob}}
\def\id {\mathrm{id}}
\def\CA{\mathcal{A}}

\newtheorem{thm}{Theorem}[section]
\newtheorem{cor}[thm]{Corollary}
\newtheorem{lemma}[thm]{Lemma}

\newtheorem{definition}[thm]{Definition}
\newtheorem{example}[thm]{Example}
\newtheorem{remark}[thm]{Remark}

\newtheorem{proposition}[thm]{Proposition}

\theoremstyle{definition}

\begin{document}
\title[Intermediate Subalgebras in tensor products]{Non-commutative Intermediate Factor theorem associated with $W^*$-dynamics of product groups}
\author[Amrutam]{Tattwamasi Amrutam}
\address{Institute of Mathematics of the Polish Academy of Sciences, ul. Sniadeckich 8, 00-656, Warszawa, Poland}
\email{tattwamasiamrutam@gmail.com}
\author[Jiang]{Yongle Jiang}
\address{School of Mathematical Sciences, Dalian University of Technology, Dalian, 116024, China}
\email{yonglejiang@dlut.edu.cn}
\author[Zhou]{Shuoxing Zhou}\address{\'Ecole Normale Sup\'erieure\\ D\'epartement de math\'ematiques et applications\\ 45 rue d'Ulm\\ 75230 Paris Cedex 05\\ FRANCE}
\email{shuoxing.zhou@ens.psl.eu}
\date{\today}
\begin{abstract}
Let $G = G_{1} \times G_{2}$ be a product of two locally compact, second countable groups and 
$\mu \in \mathrm{Prob}(G)$ be of the form $\mu = \mu_{1} \times \mu_{2}$, 
where $\mu_{i} \in \mathrm{Prob}(G_{i})$. Let $(B,\nu_B)$ be the associated Poisson boundary. 
We show that every intermediate $G$-von Neumann algebra $\cM$ with
\[
\cN \subseteq \mathcal{M} \subseteq \cN \,\bar{\otimes}\, L^{\infty}(B,\nu)
\]
splits as a tensor product of the form 
$\cN\bar{\otimes}L^{\infty}(C,\nu_C)$, where $(C,\nu_C)$ is a $(G,\mu)$-boundary. Here, $\cN$ is a tracial von Neumann algebra on which $G$ acts trace-preservingly. This generalizes the Intermediate Factor Theorem proved by Bader--Shalom (\cite[Theorem~1.9]{BS06}) in the measurable setup. 

In addition, we give various other examples of the splitting phenomenon associated 
with $W^{*}$-dynamics. We also show that certain assumptions are necessary for 
the intermediate algebras to split, and ideals in the ambient tensor product algebra obstruct the splitting phenomenon. We also use the Master theorem from \cite{glasner2023intermediate} to resolve the second part of \cite[Problem~5.2]{jiangskalski} in the affirmative.
\end{abstract}

\maketitle

\tableofcontents
\section{Introduction}

Rigidity is one of the central themes at the intersection of ergodic theory, group actions, and operator algebras. 
At its broadest level, rigidity phenomena assert that highly symmetric dynamical systems have very limited intermediate structure, in the sense that they are often forced to split as products. 
The paradigm was already visible in Margulis’ celebrated normal subgroup theorem and the subsequent works of Nevo-Stuck-Zimmer~\cite{stuck1994stabilizers,nevo1999homogenous,nevo2002generalization}, where measurable cocycles and actions of higher rank lattices exhibit remarkable structural constraints, i.e., for certain Lie groups: when a semisimple Lie group acts on its Furstenberg boundary, any intermediate factor between the base space and the boundary extension is forced to split as a product.

These theorems, known as \textit{intermediate factor theorems}, lie at the core of measurable rigidity. The common principle is that boundaries, equipped with their strong ergodic and stationarity properties, leave no room for exotic intermediate structures.  

\medskip

In 2006, Bader and Shalom \cite{BS06} introduced and proved a striking new version of this principle: the \textit{Intermediate Factor Theorem} (IFT) for product groups. 
Let $G=G_1\times G_2$ be a product of two locally compact second countable groups, and let $(B,\nu)$ be the $(G,\mu)$-Poisson boundary for an admissible measure $\mu$. 
Bader--Shalom proved that if $(X,\xi)$ is an ergodic measure-preserving $G$-space, then every $G$-invariant intermediate factor
\[(X,\xi)\times (B,\nu)
 \ \to \ (Z,\zeta) \ \to (X,\xi)\ 
\]
splits as $(Z,\zeta)\cong (X,\xi)\times(C,\nu_C)$ for some $(G,\mu)$-boundary $(C,\nu_C)$. 
This result (henceforth called the \textit{Bader--Shalom IFT}) has far-reaching consequences in the context of normal subgroups of irreducible lattices in product groups. 
The Bader--Shalom theorem should be seen as the measurable cornerstone of the splitting principle. 

The present paper should be viewed as a \textit{spiritual successor} to \cite{amrutam2025splitting}, but in the von Neumann algebraic setting. 
Our central goal is to establish a \textit{noncommutative Intermediate Factor Theorem} (NC-IFT), which extends the measurable IFT of Bader--Shalom to tracial von Neumann algebras. 
In this framework, the role of the base probability space is played by a tracial von Neumann algebra $(\cN,\tau)$, and $G$ acts by trace-preserving automorphisms.
Our first main result is the following.
\begin{thm}
\label{thm:mainNCIFT}
Assume that $(G, \mu) = (G_1 \times G_2, \mu_1 \times \mu_2)$, where $(G_i,\mu_i)$ is a lcsc group with an admissible probability measure $(i=1,2)$.
Let $(B,\nu_B)$ be the $(G,\mu)$-Poisson boundary and $(\cN,\tau_\cN)$ be a trace preserving $G$-von Neumann algebra with separable predual. Suppose $(\cN, \tau_\cN)$ is $G_i$-ergodic for each $i=1,2$.  
Then for any $G$-invariant intermediate von Neumann algebra $\cM$ with
\[
\cN \subset \cM \subset \cN\bar{\otimes}L^\infty(B, \nu_B),
\]
the algebra $\cM$ splits, that is,
$$\cM=\cN\bar{\otimes} L^\infty(C,\nu_C)$$
for some $(G,\mu)$-boundary $(C,\nu_C)$.
\end{thm}
We highlight a key ingredient in the proof of Theorem~\ref{thm:mainNCIFT}, which may be of independent interest. Let $G$ denote a locally compact second countable group,  and $\mu\in\Prob(G)$, an admissible probability measure. Let $\cM_\RN$ denote the \textbf{$G$-Radon-Nikodym factor} of $(\cM,\tau)$ (see Definition~\ref{def:rnfactor}). We show that $(\cM_\RN,\tau)$ is a $ G$-invariant von Neumann subalgebra of $\cM$ which is the image of a $ G$-equivariant conditional expectation.   
\begin{thm}
\label{thm:keyingredient}
Let $(\mathcal{M},\tau)$ be a $G$-von Neumann algebra with a not necessarily $G$-invariant trace. Then, $\cM_\RN$ is a $G$-invariant von Neumann subalgebra of $\cM$.  Moreover, the canonical $\tau$-preserving conditional expectation $\mathbb{E}_\RN:(\cM,\tau)\to(\cM_\RN,\tau)$ is $G$-equivariant.    
\end{thm}
From the non-commutative perspective, several remarkable papers have also established such rigidity, starting with Ge-Kadison, Zsido~\cite{ge1996tensor,zsido2000criterion}, to more recent remarkable non-commutative Nevo-Zimmer theorems and their generalizations~\cite{boutonnet2021stationary,bader2022charmenability, bader2023charmenability}. From a different direction, Boutonnet–Houdayer~\cite{boutonnet2023noncommutative} proved a non-commutative intermediate factor theorem for lattices in product groups by analyzing crossed-product von Neumann algebras. More recently, Glasner-Weiss~\cite{glasner2023intermediate} has established a \say{Master theorem} which specifies conditions under which intermediate algebras split. Using it, we prove an intermediate factor theorem in a concrete case involving $SL_2(\mathbb{Z})$. 
\begin{thm}
\label{thm:introtoricaction}
Let $G=SL_2(\mathbb{Z})\curvearrowright (Y,\nu)=(\mathbb{T}^{G},\mu_0^G)$ be the Bernoulli shift and $G\curvearrowright (X,\mu)=(\mathbb{T}^2,\mu)$ the standard action. Let $G\curvearrowright (Q,\eta)$ be an intermediate factor between $G\curvearrowright (Y,\nu)\times(X,\mu)$ and $G\curvearrowright (Y,\nu)$. Then there exists a factor map $G\curvearrowright (X,\mu)\to (Z,\zeta)$ such that $G\curvearrowright (Q,\eta)$ is measurably conjugate to $G\curvearrowright (Y,\nu)\times(Z,\zeta)$. 
\end{thm}
The second-named author, along with Skalski, initiated a discussion in \cite{jiangskalski} on maximal Haagerup subalgebras, where several concrete examples of maximal von Neumann subalgebras were presented. Using Theorem~\ref{thm:introtoricaction}, we resolve the second part of \cite[Problem~5.2]{jiangskalski}. In particular, we show the following. 
\begin{cor} 
\label{cor:mainmaximalhaagerup}
Let $G=SL_2(\mathbb{Z})\curvearrowright (Y,\nu)=(\mathbb{T}^{G},\mu_0^G)$ be the Bernoulli shift and $G\curvearrowright (X,\mu)=(\mathbb{T}^2,\mu)$ the standard action. Then, $L^\infty(Y,\nu)\rtimes G$ is a maximal Haagerup subalgebra of $L^\infty((Y,\nu)\times(X,\mu))\rtimes G$. 
\end{cor}
While much of this paper emphasizes situations in which intermediate subalgebras split, it is equally important to understand the mechanisms that obstruct splitting. Historically, the presence of invariant ideals is an obstruction to splitting.
We show that invariant ideals in the ambient tensor product can obstruct splitting.
\begin{thm}(Theorem \ref{thm:gensplitidealvna})
\label{thm:obstrucmain} Let $(\mathcal{M},\tau_{\mathcal{M}})$ and $(\mathcal{N},\tau_{\mathcal{N}})$ be trace-preserving $G$-von Neumann algebras and $\mathcal{P}=\mathcal{M}\bar\otimes\mathcal{N}$. Assume that $G\curvearrowright(\mathcal{M},\tau_{\mathcal{M}})$ and $G\curvearrowright(\mathcal{N},\tau_{\mathcal{N}})$ are ergodic. Let $I\triangleleft \mathcal{P}$ be a non-trivial $G$-invariant ideal of the form $q\mathcal{P}$ for some $q\in\mathcal{Z}(\mathcal{P})$. If $(\mathcal{M},\tau_{\mathcal{M}})=L^\infty(X,\nu)$ is abelian and $(\mathcal{N},\tau_{\mathcal{N}})$ non-commutative, then $\mathcal{V}_{I,\cM}=\cM+q\cP$ does not split. 
\end{thm}
In other words, when one factor is abelian while the other is non-commutative, invariant ideals can generate intermediate algebras which do not split. This shows that the splitting principle has genuine limitations and that ideals represent natural obstructions. A direct consequence is that whenever such an invariant ideal exists, one can always build a non-splitting intermediate algebra. 

The probabilistic framework of \textit{invariant random subgroups} (IRSs), initiated by Ábert--Glasner--Virág \cite{abert2014kesten}, provided a new way of encoding intermediate objects as random invariants. This has been incorporated into the non-commutative setup by the first-named author, along with Hartman and Oppelmayer~\cite{amrutam2025amenable} and further generalized by the third-named author in \cite{zhou2024noncommutative}.  Just as IRSs encode random subgroups invariant under conjugation, IRAs encode random von Neumann subalgebras invariant under group actions. In an ongoing joint work, we study the interplay between the \textit{Intermediate Factor Theorem} and the IRAs.  
\subsection*{Organization of this paper.}In addition to this section, our paper has four other sections. In the first part of Section~\ref{sec:prelinaries}, we establish the framework under which we work. We spend a considerable time in explaining the $G$-Radon-Nikodym factor associated with a $G$-von Neumann algebra (see the discussion in Subsection~\ref{Preliminaries:RN factor} and Definition~\ref{def:rnfactor}). We also show that the associated conditional expectation onto the $G$-Radon-Nikodym factor is $G$-equivariant (Proposition~\ref{prop:conditional expectation onto RN factor}). We also prove Theorem~\ref{thm:keyingredient} while we are here. We finish this section by giving various examples of splitting under some additional assumptions (see Theorem~\ref{thm:vnsplitting} and Theorem~\ref{thm:ben}). Section~\ref{sec:genNCIFT} is devoted to the proof of Theorem~\ref{thm:mainNCIFT} in its entirety. In Section~\ref{sec:obstruction}, we study how ideals are an obstruction to the splitting of the intermediate algebras. In particular, it is here that we prove Theorem~\ref{thm:obstrucmain}. Finally, we establish Theorem~\ref{thm:introtoricaction} and use it to establish Corollary~\ref{cor:mainmaximalhaagerup} in Section~\ref{sec:SL_2(Z)}.

\subsection*{Acknowledgments}  The authors are grateful for helpful discussions with Yair Hartman, Cyril Houdayer, Mehrdad Kalantar, Hanfeng Li, Zhenxing Lian, and Runlin Zhang.
Y.J. is partially supported by National Natural Science Foundation of
 China (Grant No. 12471118).

\section{Preliminaries}\label{sec:prelinaries} The statements and proofs of our results rely on tools from ergodic theory and operator algebras. In this section, we review the necessary definitions and set up the framework within which our arguments will take place.
\subsection*{Convenience} 
Throughout this paper, we assume that all von Neumann algebras have separable preduals, except in the proof of Theorem \ref{thm:NC IFT}, where enveloping von Neumann algebras are considered. When we write $L^\infty(X,\nu_X)$ for a $G$-space $(X,\nu_X)$, we mean the function algebra together with the state given by $\nu_X$. The same convention applies to crossed products $(M,\tau)\rtimes G$, endowed with the state $(\Sigma_{g\in G} a_g\lambda_g\mapsto \tau(a_e))$, and to tensor products $(M,\tau_M)\bar{\otimes}(N,\tau_N)$, endowed with the state $\tau_M\otimes\tau_N$.

\subsection{Poisson boundaries and stationary states}\label{Preliminaries:RN factor:Poisson boundaries}
Let $G$ be a locally compact second countable (lcsc) group with left Haar measure $m_G$, and $\mu\in\Prob(G)$ be an admissible Borel probability measure, i.e., $\mu$ is absolutely continuous with respect to the left Haar measure $m_G$ and $\cup_{n\geq 1}(\mathrm{supp\,} \mu)^n=G$. 

We recall the notion of a $(G,\mu)$-Poisson boundary, which is firstly introduced in \cite{Fur63a,Fur63b}. Define the \textbf{averaging Markov operator} associated with $\mu$:
$$P_\mu: L^\infty(G,m_G)\to  L^\infty(G,m_G): \ F\mapsto \left(g\mapsto\int_G F(gh)\d \mu(h)\right).$$
The space of \textbf{$(G,\mu)$-harmonic functions} is defined to be
$$\mathrm{Har}( G ,\mu)=\{F\in  L^\infty(G,m_G)|P_\mu(F)=F\}.$$
The \textbf{$(G,\mu)$-Poisson boundary} is defined to be the unique $(G,\mu)$-space $(B,\nu_B)$ such that the associated \textbf{Poisson transform}
$$P_{\nu_B}:L^\infty(B,\nu_B)\to \mathrm{Har}( G ,\mu): \ f\mapsto\left(g\mapsto\int_B f( gx)\mathrm{d} \nu_B(x)\right)$$   
is a completely isometric isomorphism.

The following theorem is \cite[Theorem 2.16]{BS06}, which also appears as \cite[Theorem 2.5]{boutonnet2021stationary}.
\begin{thm}\label{thm:boundary map}
Let $(B,\nu_B)$ be the $(G,\mu)$-Poisson boundary. Let $\cC$ be a compact convex affine $G$-space. Then for any $\mu$-stationary point $c\in\cC$ (i.e., $\int_G gc\,\d\mu(g)=c$), there exists an essentially unique $G$-equivariant measurable map $\beta:B\to \cC$ with $\int_B\beta_b\d\nu_B(b)=c$.
\end{thm}

\begin{remark}
Let $\cM$ be a $G$-von Neumann algebra and $\varphi\in \mathrm{State}(\cM)$ be a $\mu$-stationary state. Then by Theorem \ref{thm:boundary map}, there exists a $G$-equivariant measurable map $\beta:B\to \mathrm{State}(\cM)$ with $\int_B\beta_b\d\nu_B(b)=\varphi$. But note that even when $\varphi$ is normal, we do not necessarily have $\nu_B$-a.e. $\beta_b$ is normal. This is the reason why we consider the enveloping von Neumann algebra $C(B)^{**}$ in the proof of Theorem \ref{thm:NC IFT}.
\end{remark}

\begin{lemma}\label{lem:bijection stationary states and ucp maps}
Let $A$ be a $G$-C$^*$-algebra. Denote by $\mathcal{S}_\mu(A)$ the set of $\mu$-stationary states on $A$ and by $\mathrm{UCP}_G(A,L^\infty(B,\nu_B))$ the set of $G$-equivariant ucp maps from $A$ to $L^\infty(B,\nu_B)$. Then the map
$$\Phi:\mathrm{UCP}_G(A,L^\infty(B,\nu_B))\to\mathcal{S}_\mu(A):\cP\mapsto\nu_B\circ\cP$$
is a bijection.
\end{lemma}
\begin{proof}
Let $\mathrm{Har}(G,\mu)$ be the space of $\mu$-harmonic functions with the state associated with $\mu$. Then the Poisson transform $P_{\nu_B}:L^\infty(B,\nu_B)\to \mathrm{Har}(G,\mu)$ is a state preserving $G$-equivariant isomorphism between operator systems. Hence we only need to prove that
$$\Phi_0:\mathrm{UCP}_G(A,\mathrm{Har}(G,\mu))\to\mathcal{S}_\mu(A):\cP\mapsto\mu\circ\cP$$
is a bijection. For $\varphi\in\mathcal{S}_\mu(A)$, define a $G$-ucp map 
$$\cP_\varphi:A\to L^\infty(G,m_G):x\mapsto(g\mapsto \varphi(g^{-1}x)).$$
Since $\varphi$ is $\mu$-stationary, we have $\cP_\varphi(x)\in \mathrm{Har}(G,\mu)$ for any $x\in A$. Hence $\cP_\varphi\in\mathrm{UCP}_G(A,\mathrm{Har}(G,\mu))$. 

Fix $\varphi\in\mathcal{S}_\mu(A)$, then for any $x\in A$, we have
$$\mu\circ\cP_\varphi(x)=\int_G \varphi(g^{-1}x)\d\mu(g)=\mu\ast\varphi(x)=\varphi(x).$$
Hence $\varphi=\mu\circ\cP_\varphi$.

Fix $\cP\in\mathrm{UCP}_G(A,\mathrm{Har}(G,\mu))$, then for any $x\in A$ and $g\in G$,
$$\cP_{\mu\circ\cP}(x)(g)=\mu\circ\cP(g^{-1}x)=\int_G \cP(g^{-1}x)(h)\d\mu(h)=\int_G \cP(x)(gh)\d\mu(h)=\cP(x)(g).$$
The third equality holds because $\cP$ is $G$-equivariant. The last equality holds because $\cP(x)\in\mathrm{Har}(G,\mu)$.

Therefore, $\varphi\mapsto\cP_\varphi$ is exactly the inverse of $\Phi_0$, which finishes the proof.
\end{proof}

Recall that a nonsingular $G$-space $(Y,\nu_Y)$ is called \textbf{metrically ergodic} if for any separable metric space $(Z,d)$ with continuous isometric $G$-action, any $G$-equivariant measurable map $f:Y\to Z$ must be $\nu_Y$-a.e. constant. In particular, the $(G,\mu)$-Poisson boundary is always metrically ergodic \cite[Theorem 2.7]{BF14}. The following lemma also appears as part of the proof of \cite[ Proposition 4.17]{bader2022charmenability}.

\begin{lemma}
\label{lem:metric ergodic}
 Let $(\mathcal{N},\tau)$ be a trace-preserving $G$-von Neumann algebra. Let $(Y,\nu_Y)$ be a metrically ergodic $G$-space. Then, 
 $((\mathcal{N},\tau)\bar{\otimes}L^{\infty}(Y,\nu_Y))^G=\cN^G\otimes 1$.
\end{lemma}
\begin{proof}
By viewing $\mathcal{N}\bar{\otimes} L^\infty(Y)$ as $L^\infty(Y,\mathcal{N})$, a $G$-invariant element of $\mathcal{N}\bar{\otimes} L^\infty(Y)$ is equivalent to a $G$-equivariant (essentially) bounded measurable map $f:(Y,\nu_Y)\to\mathcal{N}$. View $\mathcal{N}$ as a subspace of $(L^2(\mathcal{N}),\Vert\cdot\Vert_{2,\tau})$, which is a separable metric space with continuous isometric $G$-action. Since $(Y,\nu_Y)$ is metrically ergodic, such a $G$-equivariant map $f:(Y,\nu_Y)\to(L^2(\mathcal{N}),\Vert\cdot\Vert_{2,\tau})$ must be $\nu_Y$-a.e. constant and takes a single value in $\cN^G$, which finishes our proof. 
\end{proof}

\subsection{Radon-Nikodym factors and equivariant maps}\label{Preliminaries:RN factor}
The notion of Radon-Nikodym factors was firstly introduced for discrete groups in \cite{KV83} and then extended to locally compact groups in \cite{NZ00}. Let $G$ be a locally compact second countable group and $(X,\Sigma_X,\nu_X)$ be a nonsingular $G$-space, where $\Sigma_X$ is the the $\sigma$-algebra of $(X,\nu_X)$ and $\overline{\Sigma}_X$ is the completion with respect to $\nu_X$. Following \cite[Definition 1.12]{NZ00}, the Radon-Nikodym factor $(X_\RN,\Sigma_{X_\RN},\nu_{X_\RN})$ is the unique $G$-factor of $(X,\Sigma_X,\nu_X)$ such that $\overline{\Sigma}_{X_\RN}\subset\overline{\Sigma}_X$ is the minimal $\sigma$-subalgebra that keeps the function family $\{\frac{\d g\nu_X}{\d\nu_X}\}_{g\in G}$ measurable, or equivalently, $L^\infty(X_\RN)\subset L^\infty(X)$ is the von Neumann subalgebra generated by $\{(\frac{\d g\nu_X}{\d\nu_X})^{it}\}_{g\in G,t\in\R}$. We denote the factor map by $\pi_{X_\RN}:(X,\nu_X)\to(X_\RN,\nu_{X_\RN})$. Following \cite{NZ00}, this factor map is with relatively $G$-invariant measure, i.e., the canonical trace preserving conditional expectation $E:L^\infty(X,\nu_X)\to L^\infty(X_\RN,\nu_{X_\RN})$ is $G$-equivariant.

Recently, \cite{Zh23} extended the notion of Radon-Nikodym factors to the setting of W$^*$-inclusions. Let $(M,\tau)$ be a separable tracial von Neumann algebra. By a W$^*$-inclusion $(M,\tau)\subset(\CA,\varphi_\CA)$ or a W$^*$-extension $(\CA,\varphi_\CA)$ of $(M,\tau)$, we mean that $M\subset\CA$ is a W$^*$-inclusion and $\varphi_\CA$ is a normal faithful state on $\CA$ with $\varphi|_\CA=\tau$. Let $\{\sigma_t^\CA\}_{t\in\R}$ be the automorphism group of $(\CA,\varphi_\CA)$. Then following \cite[Definition 3.4]{Zh23}, the \textbf{noncommutative Radon-Nikodym factor} $(\CA,\varphi_\CA)_\RN=(\CA_\RN,\varphi_{\CA_\RN})$ of $(\CA,\varphi_\CA)$ is the von Neumann subalgebra $\CA_\RN\subset\CA$ generated by $\{\sigma_t^\CA(z)\mid z\in M,t\in \R\}$ with the state $\varphi_{\CA_\RN}=\varphi_\CA|_{\CA_\RN}$. Since $(\CA_\RN,\varphi_{\CA_\RN})\subset(\CA,  \varphi_\CA)$ is a von Neumann subalgebra that is invariant under the action of automorphism group, by \cite[Theorem IX.4.2]{TakesakiBookII}, there always exists a normal faithful state preserving conditional expectation $E:(\CA,\varphi_\CA)\to(\CA_\RN,\varphi_{\CA_\RN})$.

Assume that $G$ is discrete. It is shown in \cite[Example 3.5]{Zh23} that the noncommutative Radon-Nikodym factor of $L^\infty(X,\nu_X)\rtimes G$ with respect to $L(G)$ is exactly $L^\infty(X_\mathrm{RN},\nu_{X_\mathrm{RN}})\rtimes G$.

We now fix a locally compact second countable group $G$ and an admissible probability measure $\mu\in\Prob(G)$. Let us define the Radon-Nikodym factor in the setting of tracial W$^*$-dynamics. 
\begin{definition}
\label{def:rnfactor}
Let $(\cM,\tau)$ be a (not necessarily trace preserving) $G$-tracial von Neumann algebra. Following \cite[Corollary 3.6]{TakesakiBookII}, for $g\in G$, let 
$$h_g:=(D(\tau\circ g):D\tau)\in L^1(\cM,\tau).$$
We define the \textbf{$G$-Radon-Nikodym factor} of $(\cM,\tau)$ to be the von Neumann subalgebra $\cM_\RN:=\langle h_g^{it}\rangle_{g\in G, t\in\R}\subset \cM$. Here $\langle\,\cdot \,\rangle$ refers to the von Neumann algebra generated by the given set.
\end{definition}

The following Propositions \ref{prop:Mrn=Zrn} and \ref{prop:Arn=Mrn} show the relationship between the newly defined Radon–Nikodym factor for tracial W$^*$-dynamics and the known ones for $G$-spaces and W$^*$-inclusions. 
\begin{proposition}\label{prop:Mrn=Zrn}
With the same conditions above, assume that $(Z(\cM),\tau|_{Z(\cM)})=L^\infty(Z,\nu_Z)$ for a nonsingular $G$-space $(Z,\nu_Z)$. Then $\cM_\RN=L^\infty(Z_\RN,\nu_{Z_\RN})\subset L^\infty(Z,\nu_Z)$. In particular, $\cM_\RN$ is a $G$-invariant subalgebra of $\cM$.
\end{proposition}
\begin{proof}
For any $g\in G$, since $\tau\circ g=\tau(h_g\,\cdot\,)$ is still a trace on $\cM$, we have $h_g\in \cM'\cap L^1(\cM,\tau)=L^1(Z,\nu_Z)$. Since $\tau|_{L^\infty(Z)}=\nu_Z$, we also have $\nu_Z\circ g=\nu_Z(h_g\,\cdot\,)$ on $L^\infty(Z,\nu_Z)$. Hence $h_g=\frac{\d g^{-1}\nu_Z}{\d\nu_Z}\in L^1(Z,\nu_Z)$ for any $g\in G$. Therefore, $\cM_\RN=\left\langle \left(\frac{\d g^{-1}\nu_Z}{\d\nu_Z}\right)^{it}\right\rangle_{g\in G,t\in \R}=L^\infty(Z_\RN,\nu_{Z_\RN})$.
\end{proof}

\begin{proposition}\label{prop:Arn=Mrn}
With the same conditions above, further assume that $G$ is discrete. Then as a W$^*$-extension of $L(G)$, the Radon-Nikodym factor of $(\cM,\tau)\rtimes G$ is exactly $(\cM_\RN,\tau)\rtimes G$.
\end{proposition}
\begin{proof}
The proof is basically the same as \cite[Example 3.5]{Zh23}. Let $(M,\tau_M)=L(G)$ and $(\CA,\varphi_\CA)=(\cM,\tau)\rtimes G$. Then by \cite[Theorem X.1.17]{TakesakiBookII}, the automorphism group $(\sigma_t^\CA)_{t\in\R}$ of $(\CA,\varphi_\CA)$ satisfies that for $\lambda_g\in L(G)$ $(g\in G)$,
$$\sigma_t^\CA(\lambda_g)=\lambda_g\cdot (D(\tau\circ g):D\tau)^{it}=\lambda_g\cdot h_g^{it}.$$
Therefore, $\CA_\RN=\langle \sigma^\CA_t(\lambda_g)\rangle_{g\in G,t\in\R}=\langle \lambda_g,h_g^{it}\rangle_{g\in G,t\in\R}=\cM_\RN\rtimes G$.
\end{proof}

\begin{proposition}\label{prop:conditional expectation onto RN factor}
With the same conditions above, the canonical $\tau$-preserving conditional expectation $\mathbb{E}_\RN:(\cM,\tau)\to(\cM_\RN,\tau)$ is $G$-equivariant.
\end{proposition}
\begin{proof}
Fix a countable dense subgroup $G_0<G$ and view it as a discrete group. By \cite[Proposition 1.14]{NZ00} and Proposition \ref{prop:Mrn=Zrn}, we have $\cM_\RN=\langle h_g^{it}\rangle_{g\in G, t\in\R}=\langle h_{g_0}^{it}\rangle_{g_0\in G_0, t\in\R}$. Hence $\cM_\RN$ is also the $G_0$-Radon-Nikodym factor of $\cM$. By Proposition \ref{prop:Arn=Mrn}, for the W$^*$-inclusion $L(G_0)\subset(\cM,\tau)\rtimes G_0$, the Radon-Nikodym factor of $(\cM,\tau)\rtimes G_0$ is exactly $(\cM_\RN,\tau)\rtimes G_0$. Hence there exists a state preserving conditional expectation $\mathbb{E}:(\cM,\tau)\rtimes G_0\to(\cM_\RN,\tau)\rtimes G_0$. Since $\mathbb{E}|_{L(G_0)}=\id_{L(G_0)}$, $\mathbb{E}$ is $G_0$-equivariant. Let $\mathbb{E}_0:(\cM_\RN,\tau)\rtimes G_0\to(\cM_\RN,\tau)$ be the canonical conditional expectation onto $\cM_\RN$. Then $\mathbb{E}_0$ is also state preserving and $G_0$-equivariant. Consider $\mathbb{E}_0\circ \mathbb{E}|_\cM:(\cM,\tau)\to(\cM_\RN,\tau)$, which is $\tau$-preserving and $G_0$-equivariant. By the uniqueness of $\tau$-preserving conditional expectation onto $\cM_\RN$ \cite[Theorem IX.4.2]{TakesakiBookII}, we must have $\mathbb{E}_\RN=\mathbb{E}_0\circ E|_\cM$. Therefore, $\mathbb{E}_\RN$ is $G_0$-equivariant. Moreover, by the arbitrariness of $G_0$, $\mathbb{E}_\RN$ is $G$-equivariant.
\end{proof}
Combining Proposition~\ref{prop:Mrn=Zrn} along with Proposition~\ref{prop:conditional expectation onto RN factor} we obtain the following corollary as an immediate consequence. This is Theorem~\ref{thm:keyingredient} from the introduction. 
\begin{cor}
Let $G$ denote a locally compact second countable group,  and $\mu\in\Prob(G)$, an admissible probability measure. Given a $G$-von Neumann algebra $(\cM,\tau)$ with a not-necessarily $G$-invariant trace, the $G$-Radon-Nikodym factor $\cM_\RN$ is a $G$-invariant von Neumann subalgebra of $\cM$.  Moreover, the canonical $\tau$-preserving conditional expectation $$\mathbb{E}_\RN:(\cM,\tau)\to(\cM_\RN,\tau)$$ is $G$-equivariant.     
\end{cor}

The following lemma is a generalization of \cite[Proposition 2.6]{BS06}.
\begin{lemma}\label{lem:h G-inv}
Let $\cM$ be a $G$-von Neumann algebra. Assume that $\tau_1$ and $\tau_2$ are both normal faithful $\mu$-stationary traces on $\cM$. Let $h=(D\tau_2:D\tau_1)\in L^1(\cM,\tau_1)$. Then $h$ must be $G$-invariant.
\end{lemma}
\begin{proof}
Assume that $Z(M)=L^\infty(Z)$ and $\tau_i|_{L^\infty(Z)}=\nu_i\in\Prob(Z)$ $(i=1,2)$. Since $\tau_2=\tau_1(h \,\cdot\,)$ is still a trace, we have $h\in \cM'\cap L^1(\cM,\tau_1)=L^1(Z,\nu_1)$. Hence $h=\frac{\d\nu_2}{\d \nu_1}$. Since $\tau_1$ and $\tau_2$ are $\mu$-stationary, so are $\nu_1$ and $\nu_2$. Then following \cite[Proposition 2.6]{BS06}, we know that $h$ is $G$-invariant.
\end{proof}

Given an inclusion $(\mathcal{N},\tau|_{\cN})\subset(\mathcal{M},\tau)$, we denote by $\mathbb{E}_{\cN}^{\cM}$ the canonical $\tau$-preserving conditional expectation onto $\cN$ from $\cM$. The following generalizes \cite[Lemma~2.2]{BS06}.
\begin{proposition}
\label{prop:commutingequivariance}
Let $(\mathcal{M},\tau)$ be a $G$-von Neumann algebra, where $\tau$ is a (not necessarily $G$-invariant) normal faithful state on $\cM$. Assume that we have the following inclusion of $G$-von Neumann algebras
\[(\mathcal{P},\tau|_{\mathcal{P}})\subset(\mathcal{N},\tau|_{\cN})\subset(\mathcal{M},\tau)\]
with the property that there exists a $\tau$-preserving conditional expectation onto each of these subalgebras. Then, $\mathbb{E}_{\mathcal{P}}^{\cM}$ is $G$-equivariant if and only if $\mathbb{E}_{\mathcal{N}}^{\cM}$ and $\mathbb{E}_{\mathcal{P}}^{\cN}$ are $G$-equivariant. 
\begin{proof} By the uniqueness of $\tau$-preserving conditional expectations \cite[Theorem IX.4.2]{TakesakiBookII}, we see that $\mathbb{E}_{\mathcal{P}}^{\cM}=\mathbb{E}_{\mathcal{P}}^{\cN}\circ\mathbb{E}_{\cN}^{\cM}$. Consequently, it follows that if $\mathbb{E}_{\mathcal{P}}^{\cN}$ and $\mathbb{E}_{\cN}^{\cM}$ are $G$-equivariant, then $\mathbb{E}_{\mathcal{P}}^{\cM}$ is $G$-equivariant. Now, assume that $\mathbb{E}_{\mathcal{P}}^{\cM}$ is $G$-equivariant. This implies that $\mathbb{E}_{\mathcal{P}}^\cN=\mathbb{E}_{\mathcal{P}}^\cM|_{\cN}$ is $G$-equivariant. Therefore, we see that
\[
\mathbb{E}_{\mathcal{P}}^{\cM}= g \cdot \mathbb{E}_{\mathcal{P}}^{\cM} \cdot g^{-1}= g \mathbb{E}_{\mathcal{P}}^\cN g^{-1}\circ g \mathbb{E}_{\cN}^{\cM} g^{-1}=\mathbb{E}_{\mathcal{P}}^\cN\circ
g \mathbb{E}_{\cN}^{\cM} g^{-1}.\]
Applying $\tau$ on both sides, we get that
\[\tau=\tau\circ \mathbb{E}_{\mathcal{P}}^{\cM}=\tau \circ \mathbb{E}_{\mathcal{P}}^\cN\circ
g \mathbb{E}_{\cN}^{\cM} g^{-1}=\tau\circ
g \mathbb{E}_{\cN}^{\cM} g^{-1}.
\]
Therefore, $g \mathbb{E}_{\cN}^{\cM} g^{-1}$ is a $\tau$-preserving conditional expectation from $\cM$ onto $\cN$. By the uniqueness of $\tau$-preserving conditional expectations, we get that
\[g \mathbb{E}_{\cN}^{\cM} g^{-1}=\mathbb{E}_{\cN}^{\cM}.
\]
The claim follows.
    \end{proof}    
\end{proposition}

The following is an immediate corollary of Propositions \ref{prop:conditional expectation onto RN factor} and \ref{prop:commutingequivariance}.
\begin{proposition}
Let $(\cM,\tau)$ be a (not necessarily trace preserving) $G$-tracial von Neumann algebra. Assume that $\mathcal{N}\subset\mathcal{M}$ is a $G$-invariant subalgebra. Suppose that, for $g\in G$, we have
$$h_g:=(D(\tau\circ g):D\tau)\in L^1(\cN,\tau).$$ Then, the $\tau$-preserving conditional expectation $\mathbb{E}_{\mathcal{N}}:\mathcal{M}\to\mathcal{N}$ is $G$-equivariant.    
\end{proposition}

Let $G=G_1\times G_2$ and $\mu=\mu_1\times \mu_2$ as in the theorem. For a $(G,\mu)$-von Neumann algebra $(\cN,\tau)$-that is, one satisfying $\mu*\tau=\tau$. We denote by $\cN^{G_i}$ the subalgebra of $G_i$-invariant elements. The result below extends \cite[Proposition~1.10]{BS06} in two directions: first, by working with arbitrary $G$-von Neumann algebras (not necessarily tracial), and second, by removing the assumption of ergodicity.
\begin{proposition}\label{prop: =key ingredient in the introduction}
Let $(\mathcal{N},\tau)$ be a $G$-von Neumann algebra, where $G=G_1\times G_2$. Suppose that every von Neumann subalgebra of $\cN$ is the image of a unique $\tau$-preserving conditional expectation (in particular, when $\tau$ is a trace). Moreover, assume that $\tau$ is $\mu$-stationary, where $\mu=\mu_1\times\mu_2$. Then, the von Neumann algebra generated by $\mathcal{N}^{G_1}$ and $\mathcal{N}^{G_2}$ (denoted by $\mathcal{T}$) is a $G$-invariant subalgebra of $\mathcal{N}$. Moreover, the canonical conditional expectation  
\[\mathbb{E}_{\mathcal{T}}^{\cN}:\mathcal{N}\xrightarrow[]{} \mathcal{T}\]
is $G$-equivariant.
\end{proposition}
\begin{proof}
Since $\mathcal{N}^{G_i}$ is a $G$-von Neumann algebra, the von Neumann algebra generated by them is automatically $G$-invariant. We now claim that $\tau$ is $\mu_i$-stationary for each $i=1,2$. Since $(\cN,\tau)$ is $\mu$-stationary, by passing to a compact model without any loss of generality, using Lemma~\ref{lem:bijection stationary states and ucp maps}, there exists a $\tau$-preserving $G$-ucp map $\Phi: \cN \to L^\infty(B_1,\nu_{B_1}) \bar{\otimes} L^\infty(B_2,\nu_{B_2})$. Composing it with $\text{id} \otimes \nu_i$,  we get a $\tau$- preserving $G_i$-ucp map from $\mathcal{N}$ to $L^\infty(B_i,\nu_{B_i})$. Again, by Lemma~\ref{lem:bijection stationary states and ucp maps}, it follows that $\tau$ is $\mu_i$-stationary for each $i=1,2$.

Let $(G,\mu)$ and $\cN$ be as above. For each $i=1,2$, define the convolution ucp map $T_{\mu_i}: \cN\to \cN$ by
\[
T_{\mu_i}(x)
= \int_{G_i} \sigma_{g}^{-1}(x)\, d\mu_i(g).
\]
Since $\mu_i * \tau = \tau$, we have $\tau \circ T_{\mu_i} = \tau$. 
As $\tau \in \cN_*$ is faithful, this implies that $T_{\mu_i}:\cN\to \cN$ is
a faithful normal ucp map. Next, choose a non-principal ultrafilter
$\omega \in \beta(\mathbb N)\setminus \mathbb N$ and define
\[
\mathbb{E}_{\mu_i}:= 
\lim_{n\to \omega} \frac{1}{n} \sum_{k=1}^n T_{\mu_i}^k(x).
\]
The limit is taken in the ultra-weak topology. Observe that $\mathbb{E}_{\mu_i}$ 
is a ucp map on $\cN$, it is idempotent, and its image is the set of 
elements invariant under $T_{\mu_i}$. This construction has already been used in \cite[Proposition~2.7]{bader2022charmenability}. It follows from \cite[Lemma~2.8]{bader2022charmenability} that $\mathbb{E}_{\mu_1}$ is $G_2$-equivariant and $\mathbb{E}_{\mu_2}$ is $G_1$-equivariant. Since $\tau\circ\mathbb{E}_{\mu_i}=\tau$, it follows that $\mathbb{E}_{\mu_i}=\mathbb{E}_{\cN^{G_i}}^{\cN}$ for each $i=1,2$. Now, observe that we have the following commuting diagram for each $i=1,2$.
\[
\begin{tikzcd}[column sep=huge] % adjust spacing between columns
\mathcal{N} 
  \arrow[r, "\;\mathbb{E}_{\mathcal{T}}^{\cN}\;"] 
  \arrow[rr, bend right=20, red, swap, "\;\mathbb{E}_{\mathcal{N}^{G_i}}^{\mathcal{N}}\;"] 
& \mathcal{T} 
  \arrow[r, "\;\mathbb{E}_{\mathcal{N}^{G_i}}^{\mathcal{T}}\;"] 
& \mathcal{N}^{G_i}
\end{tikzcd}
\]
Therefore, using Proposition~\ref{prop:commutingequivariance} twice separately for $G_1$ and $G_2$, we obtain that the canonical conditional expectation  $\mathbb{E}_{\mathcal{T}}^{\cN}$ is $G$-equivariant.   
\end{proof}
\subsection{Master theorem of Glasner-Weiss}
More recently, Glasner and Weiss \cite{glasner2023intermediate} proved a \say{master theorem} that controls intermediate factors in measurable dynamics. We do a similar analysis for commutative subalgebras inside tensor products of tracial von Neumann algebras. Assume that $(\mathcal{M},\tau_{\mathcal{M}})$ and $(\mathcal{N},\tau_{\mathcal{N}})$ are $\Gamma$-von Neumann algebras such that the action is trace-preserving and ergodic. 
Consider $$\mathcal{N}\subset\mathcal{Q}\subset\mathcal{M}\otimes\mathcal{N}.$$

Let us, for the time being, assume that $\mathcal{N}$ is commutative, i.e., $\mathcal{N}=L^{\infty}(Y,\nu)$ for some pmp ergodic action $\Gamma\curvearrowright (Y,\nu)$. Also assume that $\mathcal{Q}=L^{\infty}(Q,\eta)$ for some pmp ergodic action $\Gamma\curvearrowright (Q,\eta)$. 
And we have the following situation $$L^{\infty}(Y,\nu)\subset L^{\infty}(Q,\eta)\subset(\mathcal{M},\tau_{\mathcal{M}})\otimes L^{\infty}(Y,\nu).$$ 
\begin{proposition}
\label{prop:GSMaster}Let $(\mathcal{M},\tau_{\mathcal{M}})$ 
  be a $\Gamma$-von Neumann algebras such that the action is trace-preserving and ergodic.  Consider the following situation:   
$$L^{\infty}(Y,\nu)\subset L^{\infty}(Q,\eta)\subset(\mathcal{M},\tau_{\mathcal{M}})\otimes L^{\infty}(Y,\nu).$$
Then, there exists an unital separable ultra-weakly dense $\Gamma$-$C^*$-algebra $\mathcal{A}\subset\mathcal{M}$, a measure $\xi\in \text{Prob}(S(\mathcal{A}))$ and a one-to-one map $$J: Q\to Y\times S(\mathcal{A})$$ such that $J_*\eta=\nu\vee\xi$, a joining of $\nu$ and $\xi$.
\begin{proof}
Let $\mathbb{E}_Q:(\mathcal{M},\tau_{\mathcal{M}})\otimes L^{\infty}(Y,\nu)\to L^{\infty}(Q,\eta)$ denote the canonical conditional expectation. Let $\mathcal{A}$ be an ultraweakly dense unital separable $\Gamma$-invariant $C^*$-subalgebra of $(\mathcal{M},\tau_{\mathcal{M}})$ (see \cite[proposition~2.1]{amrutam2024subalgebras}). Then, $\mathbb{E}_Q|_{\mathcal{A}}:\mathcal{A}\to L^{\infty}(Q,\eta)$ is a $\Gamma$-equivariant ucp map (the map $\mathbb{E}_{Q}$ is $\Gamma$-equivariant to start with). Correspondingly, dualizing it, we obtain a $\Gamma$-equivariant measurable map $\varphi: Q\to S(\mathcal{A})$  defined by
$$\varphi(q)(a)=\mathbb{E}_{Q}(a)(q), ~q\in Q\text{ and } a\in\mathcal{A}.$$
A proof for the above can be found in \cite{bassi2020separable}.
Obviously, for $a\in\mathcal{A}$, we have
$$\int_Q\varphi(q)(a)d\eta(q)=\eta\circ\mathbb{E}_Q(a)=\tau_{\mathcal{M}}(a)$$

Let $\theta: Q\to Y$ denote the corresponding factor map. So, now we can define $J: Q\to Y\times S(\mathcal{A})$ defined by $J(q)=(\theta(q), \varphi(q))$. We claim that $J$ is the required map. Let $\xi=\varphi_*\eta$.\\
\textit{Claim-1: $J$ is one-to-one}. Let us assume that $J(q)=J(q')$ for some $q,q'\in Q$. In particular, we obtain that $\varphi(q)=\varphi(q')$. Therefore, we obtain that  $\mathbb{E}_{Q}(a)(q)=\mathbb{E}_{Q}(a)(q')$ for all $a\in\mathcal{A}$. Hence, for all $f\in L^{\infty}(Y,\nu)$ and for all $a\in\mathcal{A}$, we see that \[\mathbb{E}_{Q}(a\otimes f)(q)=\mathbb{E}_{Q}(a)(q)f(\theta(q))=\mathbb{E}_{Q}(a)(q)f(\theta(q'))=\mathbb{E}_{Q}(a\otimes f)(q')\] 
Since $\mathcal{A}$ is ultraweakly dense, it follows that $q=q'$.\\
\textit{Claim-2:} $J_*\eta=\nu\vee\xi$.\\
Let $p: Y\times S(\mathcal{A})\to Y$ and $q: Y\times S(\mathcal{A})\to S(\mathcal{A})$ denote the corresponding projections. We shall show that $p_*(J_*\eta)=\nu$ and $q_*(J_*\eta)=\xi$ from whence the claim will follow. For $f\in L^{\infty}(Y,\nu)$, we have
\begin{align*}
\int_Yf(y)dp_*(J_*\eta)(y)&=\int_Yfd(p\circ J)_*\eta(y)\\&=\int_Qf(p\circ J)(q)d\eta(q)=\int_Q f(\theta(q))d\eta(q)\\&=\int_Yf(y)d(\theta_*\eta)(y)\\&=\int_Yf(y)d\nu(y).    
\end{align*}
Similarly, for any $a\in C(S(\mathcal{A}))$, we have
\begin{align*}
\int_{S(\mathcal{A})}a(\tau)dq_*(J_*\eta)(\tau)&=\int_{S(\mathcal{A})}a(\tau)d(q\circ J)_*\eta(\tau)\\&=\int_Qa(q\circ J)(q')d\eta(q')=\int_Q a(\varphi(q'))d\eta(q')\\&=\int_{S(\mathcal{A})}a(\tau)d(\varphi_*\eta)(\tau)\\&=\int_{S(\mathcal{A})}a(\tau)d\xi(\tau).    
\end{align*}
The proof is now complete. 
\end{proof}
\end{proposition}
\begin{remark}
We do not need $\tau_{\mathcal{M}}$ to be measure-preserving. We just need $\mathbb{E}_Q$ to be an equivariant map.
\end{remark}
\subsection{Some Examples of Splitting} In this subsection, we give some new examples of splitting theorems in the von Neumann setup. 
\begin{lemma}
\label{lem:factor}
Let $\Gamma$ be an ICC group. Assume that $(\mathcal{M},\tau)$ is a $(\Gamma,\mu)$-ergodic space in the sense that $\tau$ is $\mu$-stationary and $\text{Supp}(\mu)$ generates $\Gamma$. Then the action of $
\Gamma$ on the crossed product $\mathcal{M}\rtimes\Gamma$ by conjugation is also ergodic.    
\end{lemma}
\begin{proof}
The proof is essentially the same as in \cite[Lemma~2.16]{kalantar2023invariant}. 
Let $a\in\mathcal{M}\rtimes\Gamma$ be $\Gamma$-fixed, and set $b=a-\mathbb{E}(a)$, where $\mathbb{E}: \mathcal{M}\rtimes \Gamma\rightarrow \mathcal{M}$ denotes the canonical conditional expectation. Then $b$ is $\Gamma$-fixed and $\mathbb{E}(b) = 0$. For $g\in\Gamma$, let $\varphi(g):=\mathbb{E}(b\lambda(g)^{-1}) \in\mathcal{M}$. Since $b$ is fixed by $\Gamma$, the map $\Gamma\ni g\overset{\xi}{\mapsto}\|\varphi(g)\|_{\tau}^2\in\mathbb{C}$ is $\mu$-harmonic (for the conjugate action). Moreover,
$$\sum_{g\in\Gamma}\|\varphi(g)\|_{\tau}^2=\tau(\mathbb{E}(bb^*))<\infty.$$

Therefore the map $\xi:\Gamma\to\mathbb{C}, g\mapsto \|\varphi(g)\|_{\tau}^2$ is in $\ell^1(\Gamma)$. Thus, $\xi$ attains its maximum values on a finite conjugate-invariant subset of $\Gamma$, that is $\{e\}$ by the ICC assumption. Therefore, $\xi(e)=0$  and hence, $\xi = 0$. So , $b=0$  and $a=\mathbb{E}(a)$. The element $\mathbb{E}(a)\in\mathcal{M}$ is $\Gamma$-invariant, hence constant by the ergodicity of $\Gamma\curvearrowright\mathcal{M}$.
    
\end{proof}
The following is a measurable version of \cite[Proposition~1.2]{amrutam2025splitting}. In the measurable setup, the result is merely an observation from \cite{ge1996tensor} put together with Lemma~\ref{lem:factor}.
\begin{thm}
\label{thm:vnsplitting}
Let $G$ be a countable ICC group. Let $\mathcal{M}$ be a $G$-von Neumann algebra. Let $H\le G$ be such that $(\mathcal{M},\tau)$ is a $(H,\mu)$-ergodic space.  Let $\mathcal{N}$ be a $G$-von Neumann algebra such that $H\curvearrowright \mathcal{N}$ is trivial. Then, every $G$-invariant von Neumann subalgebra $\mathcal{M}\bar{\otimes} \mathbb{C}\subset\mathcal{C}\subset \mathcal{M}\bar{\otimes} \mathcal{N}$ splits. 
\begin{proof}
  This is a direct corollary of the Ge-Kadison splitting theorem. To see this, set $\mathcal{Q}=\mathcal{M}\rtimes H$. It follows from Lemma~\ref{lem:factor} that $Q$ is a factor. Thus, by considering the inclusion of von Neumann algebras: $\mathcal{Q}\subseteq \mathcal{C}\rtimes H\subseteq (\mathcal{M}\bar{\otimes} \mathcal{N})\rtimes H=\mathcal{Q}\bar{\otimes} \mathcal{N}$. Hence $\mathcal{C}\rtimes H=\mathcal{Q}\bar{\otimes} \mathcal{P}=(\mathcal{M}\bar{\otimes} \mathcal{P})\rtimes H$ for some $\mathcal{P}\subseteq \mathcal{N}$. Then $\mathcal{C}=\mathcal{M}\bar{\otimes} \mathcal{P}$.
\end{proof}
\end{thm}
\begin{example}
Let $\Gamma$ be a countable group with a non-trivial amenable radical. Let $(X,\nu)$ be a $\mu$-USB for some $\mu\in\text{Prob}(\Gamma)$ (see \cite[Definition~3.9]{HartKal} to know what $\mu$-USB means). It follows from \cite[Proposition~3.12]{HartKal} that $\text{Rad}(\Gamma)\subset\text{Ker}(\Gamma\curvearrowright(X,\nu))$. Now, let $(\mathcal{M},\tau)$ be any $G$-von Neumann algebra which is ergodic with respect to $H\le \text{Rad}(\Gamma)$. (Say, a pmp $G$-action $(Z,\xi)$ on which $H$ acts ergodically.) It follows from Theorem~\ref{thm:vnsplitting} that every $G$-invariant intermediate von Neumann subalgebra $\mathcal{C}$ of the form $$\mathcal{M}\bar{\otimes} \mathbb{C}\subset\mathcal{C}\subset \mathcal{M}\bar{\otimes} L^{\infty}(X,\nu)$$ splits.
\end{example}
An intermediate factor $(X\times Y,\mu\times\nu)\ \to\ (Z,\zeta)\ \to\ (Y,\nu)$
can be disintegrated over $(Y,\nu)$ as a measurable field $y\mapsto\mathcal A_y\subset L^\infty(X,\mu)$.
Non-product behavior means the fiber $\mathcal A_y$ depends on $y$ (a skew coupling).
If $\Aut(G\curvearrowright (Y,\nu))$ acts ergodically on $(Y,\nu)$ and globally preserves $L^\infty(Z,\zeta)$,
then the map $y\mapsto\mathcal A_y$ is invariant under an ergodic action, then
$\mathcal A_y=\mathcal A$ almost everywhere, therefore the intermediate factor splits as a product. The hypotheses exactly rule out
$y$-dependent (cocycle) twists, forcing independence over the $Y$-coordinate.
Following the arguments in \cite[Lemma 4.1]{ben-pinsker}, we show the following. 
\begin{thm}
\label{thm:ben}
Let $G\curvearrowright (X,\mathcal{B}_X, \mu)$ and $G\curvearrowright (Y, \mathcal{B}_Y, \nu)$ be two pmp actions. Let $G\curvearrowright (Z,\mathcal{B}_Z, \zeta)$ be an intermediate factor with $(X\times Y,\mu\times \nu)\rightarrow (Z,\zeta)\rightarrow (Y,\nu)$. Let  $Aut(G\curvearrowright (Y,\nu))$ be the group of all measure-preserving transformations $\phi: Y\rightarrow Y$ so that $\phi(gy)=g\phi(y)$ for a.e. $y\in Y$ and all $g\in G$.
Assume that $Aut(G\curvearrowright (Y,\nu))$ acts ergodically on $(Y,\nu)$ and fixes $L^{\infty}(Z,\zeta)$ globally, then $G\curvearrowright (Z,\zeta)$ splits as a product. 
\end{thm}
\begin{proof}
The proof of \cite[Lemma 4.1]{ben-pinsker} 
works verbatim in our setting. We include it for completeness. We identity $\mathcal{B}_Z$ as a Borel $\sigma$-subalgebra of $\mathcal{B}_X$. Write $\mathcal{C}=L^{\infty}(Z,\zeta)$. Consider the conditional expectation $\mathbb{E}_{\mathcal{C}}$ as a projection operator
\begin{align*}
\mathbb{E}_{\mathcal{C}}: L^2(X\times Y,\mu\times \nu)\rightarrow L^2(X\times Y, \mu\times \nu).
\end{align*}
For $\alpha\in Aut(G\curvearrowright (Y,\nu))$, define a unitary operator $U_{\alpha}: L^2(Y,\nu)\rightarrow L^2(Y,\nu)$ by $U_{\alpha}\xi=\xi\circ \alpha^{-1}$. Recall that if $E\subseteq B(\mathcal{H})$ for some Hilbert space $\mathcal{H}$, then \[E'=\{T\in B(\mathcal{H}): TS=ST~\text{for some}~S\in E\}.\]
As shown in the proof of  \cite[Lemma 4.1]{ben-pinsker}, we have the following fact. 

(\textbf{Fact}) If $T\in B(L^2(Y,\nu))$ and $T\in (\{m_f: f\in L^{\infty}(Y,\nu)\}\cup \{U_{\alpha}: \alpha \in Aut(G\curvearrowright (Y,\nu))\})'$, then $T\in \mathbb{C}1$.

Next, we claim that 
\begin{align}\label{claim: claim 1 in ergodicity and invariant lemma}
\mathbb{E}_{\mathcal{C}}\in (\{1\otimes m_f: f\in L^{\infty}(Y,\nu)\}\cup \{1\otimes U_{\alpha}: \alpha\in Aut(G\curvearrowright (Y,\nu))\})'.
\end{align}
Indeed, let us check that for any $f\in L^{\infty}(Y,\nu)$ and $\alpha\in Aut(G\curvearrowright (Y,\nu))$, we have 
\begin{align*}
\mathbb{E}_{\mathcal{C}}\circ (1\otimes m_f)=(1\otimes m_f)\circ \mathbb{E}_{\mathcal{C}},~~~~
\mathbb{E}_{\mathcal{C}}\circ (1\otimes U_{\alpha})=(1\otimes U_{\alpha})\circ \mathbb{E}_{\mathcal{C}}.
\end{align*}
Take any $\xi\otimes \eta\in L^2(X\times Y,\mu\times \nu)\cong L^2(X,\mu)\bar{\otimes}L^2(Y,\nu)$, we aim to show that
\begin{align*}
\mathbb{E}_{\mathcal{C}}(\xi\otimes f\eta)=(1\otimes m_f)\circ \mathbb{E}_{\mathcal{C}}(\xi\otimes \eta),\\
\mathbb{E}_{\mathcal{C}}(\xi\otimes U_{\alpha}\eta)=(1\otimes U_{\alpha})\circ \mathbb{E}_{\mathcal{C}}(\xi\otimes \eta).
\end{align*} 
By definition of $\mathbb{E}_{\mathcal{C}}$, the above is equivalent to showing that for any $\rho\in L^2(Z,\zeta)$, we have
\begin{align*}
\langle \xi\otimes f\eta-(1\otimes m_f)\circ \mathbb{E}_{\mathcal{C}}(\xi\otimes \eta), \rho \rangle=0,\\
\langle \xi\otimes U_{\alpha}\eta-(1\otimes U_{\alpha})\circ \mathbb{E}_{\mathcal{C}}(\xi\otimes \eta), \rho \rangle=0.
\end{align*}
Let us check them one by one.
\begin{align*}
\langle (1\otimes m_f)\circ \mathbb{E}_{\mathcal{C}}(\xi\otimes \eta), \rho\rangle\
&=\langle \mathbb{E}_{\mathcal{C}}(\xi\otimes \eta), (1\otimes m_f)^*\rho\rangle\\
&=\langle \mathbb{E}_{\mathcal{C}}(\xi\otimes \eta), (1\otimes m_{f^*})\rho\rangle\\
&=\langle \xi\otimes \eta, \mathbb{E}_{\mathcal{C}}((1\otimes m_{f^*})\rho)\rangle\\
&=\langle \xi\otimes\eta, (1\otimes m_{f^*})\rho\rangle~~~(\text{since $(1\otimes m_f^*)\in \mathcal{C}$ and $\rho\in L^2(Z,\zeta)$})\\
&=\langle (1\otimes m_{f^*})^*(\xi\otimes \eta), \rho\rangle\\
&=\langle (1\otimes m_f)(\xi\otimes \eta),\rho\rangle\\
&=\langle \xi\otimes f\eta, \rho\rangle.
\end{align*}
This finishes the proof of first identity. For the second one, we have
\begin{align*}
\langle (1\otimes U_{\alpha})\circ \mathbb{E}_{\mathcal{C}}(\xi\otimes \eta), \rho\rangle
&=\langle \mathbb{E}_{\mathcal{C}}(\xi\otimes \eta), (1\otimes U_{\alpha})^*\rho\rangle\\
&=\langle \xi\otimes \eta, \mathbb{E}_{\mathcal{C}}((1\otimes U_{\alpha}^*)\rho)\rangle\\
&=\langle \xi\otimes \eta, (1\otimes U_{\alpha})^*\rho\rangle~~~(\text{since $\alpha$ fixes $\mathcal{C}$ globally by assumption})\\
&=\langle (1\otimes U_{\alpha})(\xi\otimes \eta), \rho\rangle\\
&=\langle \xi \otimes U_{\alpha}\eta, \rho\rangle.
\end{align*}
This finishes the proof of the second identity. The rest proof is identical to the proof \cite[Lemma 4.1]{ben-pinsker}. Let 
$\mathcal{Aut}=Aut(G\curvearrowright (Y, \nu)).$
By \cite[Theorem IV. 5.9]{ben-pinsker}, the above implies that $\mathbb{E}_{\mathcal{C}}$ belongs to the following set:
\begin{align*}
&\overline{\text{span}(\{T\otimes S: T\in B(L^2(X,\mu)), S\in (\{m_f: f\in L^{\infty}(Y,\nu)\}\cup \{U_{\alpha}: \alpha\in \mathcal{Aut}'\})}^{SOT}\\
&=B(L^2(X,\mu))\otimes \mathbb{C}1.
\end{align*}
The last equality follows from the above mentioned \textbf{Fact}. So $\mathbb{E}_{\mathcal{C}}=P\otimes 1$ for some projection $P\in B(L^2(X,\mu))$. Let $\mathcal{D}=\{U\in \mathcal{B}_X: U\times Y\in\mathcal{B}_Z\}$. We show that $P(L^2(X,\mathcal{B}_X, \mu))=L^2(X,\mathcal{D},\mu)$ from whence the proof will be done. It is clear that $P(L^2(X,\mathcal{B}_X,\mu))\supseteq L^2(X,\mathcal{D},\mu)$, so it is enough to show that $P(L^2(X,\mathcal{B}_X,\mu))\subseteq L^2(X,\mathcal{D},\mu)$. To show this, it is enough to show that if $\xi\in P(L^2(X,\mathcal{B}_X,\mu))$ and $A\subseteq \mathbb{C}$ is Borel, then $\xi^{-1}(A)\in\mathcal{D}$. Fix such a $\xi$, $A$. Since $\mathbb{E}_{\mathcal{C}}=P\otimes 1$, and $P(\xi)=\xi$, we know that $\xi\otimes 1$ is $\mathcal{B}_Z$-measurable, because $\mathcal{B}_Z$ is complete. So $1_A\circ (\xi\otimes 1)$ is $\mathcal{B}_Z$-measurable. Since $1_{\xi^{-1}(A)\times Y}=1_A\circ (\xi\otimes 1)$, it follows that $\xi^{-1}(A)\times Y$ is $\mathcal{B}_Z$-measurable, so $\xi^{-1}(A)\in\mathcal{D}$.
\end{proof}

\begin{remark}
Clearly, $L^{\infty}(Z,\zeta)$ is globally invariant under $Aut(G\curvearrowright (Y,\nu))$ is a necessary condition for $G\curvearrowright (Z,\zeta)$ to split as a product.
\end{remark}

\begin{remark}
A concrete example with $Aut(G\curvearrowright (Y,\nu))$ acting on $(Y,\nu)$ ergodically is given by the diagonal action $G\curvearrowright Y:=K^G$ defined using some pmp action $G\curvearrowright K$. Indeed, since the shift action lies in $Aut(G\curvearrowright (Y, \nu))$ and hence acts on $(Y,\nu)$ ergodically. Similar to \cite[Lemma 4.2]{ben-pinsker}, given any $G\curvearrowright (Y,\nu)$, we may view it as a $G$-factor of the diagonal action $G\curvearrowright (Y^G,\nu^G)$, which satisfies the ergodicity assumption.
\end{remark}

\begin{remark}
Note that if we have some pmp ergodic action $H\curvearrowright (Y,\nu)$ such that
\begin{itemize}
    \item Either $H$ is ICC or the action is also essentially free.
    \item This action preserves $L^{\infty}(Z,\zeta)$ globally.
\end{itemize}
Then $G\curvearrowright (Z,\zeta)$ splits as a product.
Indeed, the first bullet guarantees that $L^{\infty}(Y)\rtimes H$ is a factor. Due to the second bullet, we may consider the following inclusion:
$L^{\infty}(Y)\rtimes H\subseteq L^{\infty}(Z)\rtimes H\subseteq L^{\infty}(X)\bar{\otimes}(L^{\infty}(Y)\rtimes H)$. By Ge-Kadison's splitting theorem \cite{ge1996tensor}, we deduce that $L^{\infty}(Z)\rtimes H=L^{\infty}(W)\bar{\otimes}[L^{\infty}(Y)\rtimes H]$ for some $G$-factor $G\curvearrowright W$ of $G\curvearrowright X$. This shows that $Z$ splits as $W\times Y$. The subtle step in applications is verifying the global preservation
of $L^\infty(Z)$ by $H$ (the “position” of $L^\infty(Z)$ inside $L^\infty(X\times Y)$ is
generally opaque). This hypothesis is automatic, for instance, if $L^\infty(Z)$ is invariant
under a large symmetry group of $(Y,\nu)$ (e.g.\ $\Aut(G\curvearrowright Y)$ in the theorem),
or if $L^\infty(Z)$ is normalized by a dense subgroup of such automorphisms.
\end{remark}

\section{Complete Generalization of Bader-Shalom}
\label{sec:genNCIFT}
Before presenting the proof for the noncommutative intermediate factor theorem (NC-IFT), let us recall the definition of enveloping von Neumann algebras: Given a separable C$^*$-algebra $A$, let 
$$\pi_U=\bigoplus_{\varphi\in\mathrm{State}(A)}\pi_\varphi: A\to B\left(\bigoplus_{\varphi\in\mathrm{State}(A)} L^2(A,\varphi)\right)=B(H_U)$$
be the \textbf{universal representation} of $A$. Then the \textbf{enveloping von Neumann algebra} of $A$ is defined to be $A^{**}:=\pi_U(A)''$. In particular, any state $\varphi$ on $A$ can be uniquely extended to a normal state on $A^{**}$. Moreover, if $A$ admits a $G$-action for a lcsc group $G$, then the $G$-action can be naturally extended to $A^{**}$. For more details regarding the enveloping von Neumann algebras, we refer to \cite[Subsection 1.4]{BObook} and \cite[Section III.2]{takesaki}.

Our main result is the following. 
\begin{thm}\label{thm:NC IFT}
Assume that $(G, \mu) = (G_1 \times G_2, \mu_1 \times \mu_2)$, where $(G_i,\mu_i)$ is a lcsc group with an admissible probability measure $(i=1,2)$.
Let $(B,\nu_B)$ be the $(G,\mu)$-Poisson boundary and $(\cN,\tau_\cN)$ be a trace preserving $G$-von Neumann algebra with separable predual. Suppose $(\cN, \tau_\cN)$ is $G_i$-ergodic for each $i=1,2$.  Then for any $G$-invariant intermediate von Neumann algebra $\cM$ with
\[
\cN \subset \cM \subset \cN\bar{\otimes}L^\infty(B, \nu_B),
\]
the algebra $\cM$ splits, that is,
$$\cM=\cN\bar{\otimes} L^\infty(C,\nu_C)$$
for some $(G,\mu)$-boundary $(C,\nu_C)$.
\end{thm}
Before embarking on the proof, let us outline our proof strategy, which follows the footsteps of \cite[Theorem~1.9]{BS06}. We look at the $G$-Radon-Nikodym factors of each subalgebra. This plays the role akin to the \say{invariants product functor $F^G$} used in \cite{BS06}. As such, we show that for every intermediate algebra $\cM$, its corresponding $G$-Radon-Nikodym factor is a subalgebra of $L^{\infty}(B,\nu_B)$. It should not come as a surprise then, that the $G$-Radon-Nikodym factor of the ambient tensor product is $L^{\infty}(B,\nu_B)$ itself, and that of the tracial von Neumann algebra $\cN$ is the set of scalars. Once this is achieved, the last piece of the puzzle is to show that the intermediate algebra is indeed its $G$-Radon-Nikodym factor tensored with the underlying tracial von Neumann algebra $\cN$. This is illustrated in the following diagram.
\begin{center}
\begin{tikzcd}[column sep=huge, row sep=large]
\mathcal{N}\otimes1
  \arrow[r, hook]
  \arrow[d, "\mathbb{E}_{\mathrm{RN}}^{\cN}"']
& \mathcal{M}
  \arrow[r, hook]
  \arrow[d, "\mathbb{E}_{\mathrm{RN}}^{\cM}"']
& \mathcal{N}\,\bar\otimes\,L^\infty(B,\nu_B)
  \arrow[d, "\mathbb{E}^{\mathcal{N}\bar\otimes L^\infty(B,\nu_B)}_{\mathrm{RN}}"']
\\
\mathbb{C}
  \arrow[r, hook]
& \cM_\RN
  \arrow[r, hook]
& L^\infty(B,\nu_B)
\end{tikzcd}
\end{center}

\begin{proof}[Proof of Theorem~\ref{thm:NC IFT}]
This proof is in the same spirit as Bader-Shalom's original proof of the IFT theorem \cite[Theorem 1.9]{BS06}. Denote by $\tau$ the trace $\tau_\cN\otimes\nu_B|_\cM$ on $\cM$. Let $\cM_\RN$ be the $G$-Radon-Nikodym factor of $(\cM,\tau)$. We will show $\cM_\RN\subset L^\infty(B,\nu_B)$ first.

 Let $(B_i,\nu_{B_i})$ be the $(G_i,\mu_i)$-Poisson boundary $(i=1,2)$. Then $(B,\nu_B)=(B_1\times B_2,\nu_{B_1}\otimes\nu_{B_2})$. For $g\in G$, let $h_g=(D(\tau\circ g):D\tau)\in L^1(\cM,\tau)$. Clearly, since $\tau_\cN\otimes\nu_B=\tau_\cN\otimes(\nu_{B_1}\otimes\nu_{B_2})$ is $(G_i,\mu_i)$-stationary for each $i=1,2$, so is $\tau$. For any $g_1\in G_1$, since $G_1$ commutes with $G_2$, $\tau\circ g_1$ is still $(G_2,\mu_2)$-stationary. Hence by Lemma \ref{lem:h G-inv}, $h_{g_1}=(D(\tau\circ g_1):D\tau)$ is $G_2$-invariant and $h_{g_1}^{it}\in\cM^{G_2}\subset(\cN\bar{\otimes}L^\infty(B, \nu_B))^{G_2}$. We also have
$$\cN\bar{\otimes}L^\infty(B, \nu_B)=(\cN\bar{\otimes}L^\infty (B_1,\nu_{B_1}))\bar{\otimes} L^\infty(B_2, \nu_{B_2}).$$
Note that the $G_2$-action on $(\cN,\tau_\cN)\bar{\otimes}L^\infty (B_1,\nu_{B_1})$ is  trace preserving. Since the Poisson boundary $(B_2,\nu_2)$ is $G_2$-metrically ergodic, by Lemma \ref{lem:metric ergodic}, we have
$$(\cN\bar{\otimes}L^\infty(B, \nu_B))^{G_2}=(\cN\bar{\otimes}L^\infty (B_1,\nu_{B_1}))^{G_2}\otimes1=L^\infty(B_1,\cN)^{G_2}\otimes 1.$$
Since $G_2$ acts trivially on $B_1$ and ergodically on $\cN$, we have
$$L^\infty(B_1,\cN)^{G_2}=L^\infty(B_1,\cN^{G_2})=L^\infty(B_1,\C)=L^\infty(B_1,\nu_{B_1}).$$
Therefore, we have $\langle h_{g_1}^{it}\rangle_{g_1\in G_1,t\in\R}\subset L^\infty(B_1,\nu_{B_1})\subset L^\infty(B,\nu_B)$. For the same reason, we also have $\langle h_{g_2}^{it}\rangle_{g_2\in G_2,t\in\R}\subset L^\infty(B_2,\nu_{B_2})\subset L^\infty(B,\nu_B)$.

For any $g_1g_2\in G$ $(g_i\in G_i, i=1,2)$, by the Chain Rule \cite[Theorem VIII.3.7]{takesaki} and \cite[Formula (33) in the proof of Corollary VIII.3.6]{takesaki}, we have
$$h_{g_1g_2}^{it}=(D(\tau\circ g_1g_2):D(\tau\circ g_2))^{it}\cdot(D(\tau\circ g_2):D\tau)^{it}=g_2^{-1}(h_{g_1}^{it})\cdot h_{g_2}^{it}=h_{g_1}^{it}\cdot h_{g_2}^{it}\in L^\infty(B,\nu_B).$$
By the arbitrariness of $g_1g_2$, we have $\cM_\RN\subset L^\infty(B,\nu_B)$. Hence there exists a $(G,\mu)$-boundary $(C,\nu_C)$ with $(\cM_\RN,\tau)=L^\infty(C,\nu_C)$.

Endow $(B,\nu_B)$ with a compact metrizable model. Let $C(B)^{**}$ be the enveloping von Neumann algebra of $C(B)$. Then $\nu_B$ can be viewed as a normal state on $C(B)^{**}$ and $L^\infty(B,\nu_B)=p_{\nu_B}C(B)^{**}p_{\nu_B}$, where $p_{\nu_B}$ is the support projection of $\nu_B$. Note that since $\nu_B\in\Prob(B)$ is $G$-quasi-invariant, we have that $p_{\nu_B}$ is $G$-invariant. Take a $G$-invariant intermediate von Neumann algebra $\cN\subset\hat{\cM}\subset \cN\bar{\otimes} C(B)^{**}$ with $\hat{\cM}=\cM\oplus (\cN\otimes p_{\nu_B}^\perp)$. Then $\hat{\cM}$ is weakly separable since both $\cM$ and $\cN$ are, and $p_{\nu_B}\hat{\cM}p_{\nu_B}=\cM$ . Let $\hat{\tau}=\tau_\cN\otimes\nu_B|_{\hat{\cM}}$, which is a $\mu$-stationary trace on $\hat{\cM}$. Then $(\cM,\tau)$ is exactly the GNS construction of $(\hat{\cM},\hat{\tau})$. We denote this quotient $\ast$-homomorphism by $\Phi_\cM:(\hat{\cM},\hat{\tau})\to(\cM,\tau):x\mapsto p_{\nu_B}xp_{\nu_B}$. Since $p_{\nu_B}$ is $G$-invariant, we know that $\Phi_M$ is $G$-equivariant.

By Theorem~ \ref{thm:keyingredient}, there exists a trace preserving $G$-equivariant conditional expectation $\mathbb{E}_\RN:(\cM,\tau)\to L^\infty(C,\nu_C)$. Then $$\cP:=\mathbb{E}_\RN\circ\Phi_\cM:(\hat{\cM},\hat{\tau})\to L^\infty(C,\nu_C)\subset L^\infty(B,\nu_B)$$ is a trace preserving $G$-equivariant ucp map. Since $\hat{\tau}=\tau_\cN\otimes\nu_B|_{\hat{\cM}}$, the $G$-equivariant measurable map $\beta:B\to\mathrm{State}(\hat{\cM}):(b\mapsto\tau_\cN\otimes\delta_b|_{\hat{\cM}})$ satisfies $\int_B\beta_b\d \nu_B(b)=\hat{\tau}$. Hence we can define another trace preserving $G$-equivariant ucp map
$$\cP_0 : (\hat{\cM},\hat{\tau})\to L^\infty(B,\nu_B):x\mapsto (b\mapsto\beta_b(x)).$$
Since $\hat{\tau}=\nu_B\circ \cP=\nu_B\circ\cP_0$, by the bijection between $\mu$-stationary states and ucp maps to $L^\infty(B,\nu_B)$ (Lemma \ref{lem:bijection stationary states and ucp maps}), we must have $\cP=\cP_0$. 

Let $\pi_C:(B,\Sigma_B,\nu_B)\to(C,\Sigma_C,\nu_C)$ be the factor map and denote by $\overline{\Sigma}_C\subset\overline{\Sigma}_B$ the completed $\sigma$-algebras with respect to $\nu_B$ and $\nu_C$. Let us show that $\beta$ is measurable with respect to $\overline{\Sigma}_C$. Since $\tau_\cN\otimes\delta_b$ is normal on $\cN\bar{\otimes} C(B)^{**}$, we know that $\beta$ actually takes values in $\mathrm{State_n}(\hat{\cM})$, the set of normal states on $\hat{\cM}$. Since $\hat{\cM}=\cM\oplus (\cN\otimes p_{\nu_B}^\perp)$ is weakly separable, we can take a countable weakly dense subset $(x_n)\subset \hat{\cM}$. Then the $\sigma$-algebra of $\mathrm{State_n}(\hat{\cM})$ can be countably generated by $\{\varphi\in\mathrm{State_n}(\hat{\cM})\mid \mathrm{Re}\,\varphi(x_n)>a\}$ $(n\in \N, a\in\Q)$.

For any $n\in\N$ and $a\in\Q$, since $$(b\mapsto\beta_b(x_n))=\cP_0(x_n)=\cP(x_n)\in L^\infty(C,\nu_C),$$
we have $$\beta^{-1}(\{\varphi\in\mathrm{State_n}(\hat{\cM})\mid \mathrm{Re}\,\varphi(x_n)>a\})=\{b\in B\mid \mathrm{Re}\,\beta_b(x_n)>a\}\in\overline{\Sigma}_C.$$
By the arbitrariness of $x_n$ and $a$, we know that $\beta$ is measurable with respect to $\overline{\Sigma}_C$. Hence $\beta: B\to \mathrm{State_n}(\hat{\cM})$ can be passed to a $G$-equivariant measurable map $\eta:C\to \mathrm{State_n}(\hat{\cM})$ with $\int_C\eta_c\d \nu_C(c)=\hat{\tau}$ and $\beta=\eta\circ\pi_C$ $\nu_B$-a.e.

Therefore, for $\nu_C$-a.e. $c\in C$, there exists $b\in\pi_C^{-1}(c)$ such that $\eta_c=\tau_\cN\otimes\delta_b|_{\hat{\cM}}$. Hence we have
$$(\cN,\tau_\cN)\subset(\hat{\cM},\eta_c)\subset(\cN\bar{\otimes} C(B)^{**},\tau_\cN\otimes\delta_b).$$
Clearly, the GNS construction of $(\cN\bar{\otimes} C(B)^{**},\tau_\cN\otimes\delta_b)$ is exactly isomorphic to $(\cN,\tau_\cN)$. Hence so is the GNS construction of $(\hat{\cM},\hat{\tau})$. Denote this quotient $\ast$-homomorphism by $\Phi_c:(\hat{\cM},\eta_c)\to (\cN,\tau_\cN)$. Then we have $\Phi_c|_\cN=\id_\cN$. Also note that since $\eta:C\to\mathrm{State_n}(\hat{\cM})$ is $G$-equivariant, so is the map $C\ni c\mapsto \Phi_c\in \mathrm{Hom}(\hat{\cM},\cN)$. Here $G$ acts on $\mathrm{Hom}(\hat{\cM},\cN)$ by $g(\Phi_0)=g\circ\Phi_0\circ g^{-1}$.

Define $\hat{\Phi}:\hat{\cM}\to \cN\bar{\otimes} L^\infty(C,\nu_C)=L^\infty(C,\cN)$ by
$$\hat{\Phi}(x)(c)=\Phi_c(x)\ (x\in\hat{\cM},c\in C).$$
Since $\nu_C$-a.e. $\Phi_c$ is a $\ast$-homomorphism, so is $\hat{\Phi}$. Since $c\mapsto\Phi_c$ is $G$-equivariant, we have that for any $g\in G$, $x\in \hat{\cM}$ and $\nu_C$-a.e. $c\in C$,
$$\hat{\Phi}(gx)(c)=\Phi_c(gx)=g \circ(g^{-1}\circ\Phi_c\circ g)(x)=g \circ\Phi_{g^{-1}c}(x)=g(\hat{\Phi}(x)(g^{-1}c)).$$ 
Hence $\hat{\Phi}$ is $G$-equivariant (note that $G$ acts on $L^\infty(B,\cN)$ by $g(f)=g\circ f\circ g^{-1}$). We also have that for $x\in\hat{\cM}$,
$$\tau_\cN\otimes\nu_C(\hat{\Phi}(x))=\int_C \tau_\cN\circ\Phi_c(x)\d\nu_C(c)=\int_C \eta_c(x)\d\nu_C(c)=\hat{\tau}(x).$$
Hence $\hat{\Phi}:(\hat{\cM},\hat{\tau})\to(\cN,\tau_\cN)\bar{\otimes} L^\infty(C,\nu_C)$ is trace preserving. Therefore, $\hat{\Phi}$ can be passed to the GNS construction of $(\hat{\cM},\hat{\tau})$, i.e., $(\cM,\tau)$. Let $\Phi:(\cM, \tau)\to(\cN,\tau_\cN)\bar{\otimes} L^\infty(C,\nu_C)$ be $\Phi(\Phi_\cM(x))=\hat{\Phi}(x)$ for $x\in\hat{M}$. Then $\Phi$ is a well-defined trace preserving $G$-equivariant $\ast$-homomorphism. Since $\tau$ is faithful on $\cM$, we know that $\Phi$ is injective. Since $\Phi_c|_\cN=\id_\cN$, we have $\Phi|_{\cN\otimes1}=\id_{\cN\otimes 1}$. 

Let $\tau_\cN\otimes \id: (\cN,\tau_\cN)\bar{\otimes} L^\infty(C,\nu_C)\to L^\infty(C,\nu_C)$ be the trace preserving $G$-equivariant conditional expectation onto $L^\infty(C,\nu_C)$. Then 
$$\Psi:=(\tau_\cN\otimes \id)\circ\Phi|_{L^\infty(C)}:L^\infty(C,\nu_C)\to L^\infty(C,\nu_C)\subset L^\infty(B,\nu_B)$$ 
is a trace preserving $G$-ucp map. By Lemma \ref{lem:bijection stationary states and ucp maps}, since $\Psi$ preserves $\nu_C$, we must have $\Psi=\id_{L^\infty(C)}$. Hence by \cite[Lemma 2.2]{zhou2024noncommutative}, $\Phi(L^\infty(C,\nu_C))$ is contained in the multiplicative domain of $\tau_\cN\otimes \id$, i.e., $1\otimes L^\infty(C,\nu_C)$, and $\Phi|_{L^\infty(C)}=(\tau_\cN\otimes \id)\circ\Phi|_{L^\infty(C)}=\id_{L^\infty(C)}$.

Now we know that restrictions of $\Phi$ on $\cN$ and $L^\infty(C,\nu_C)$ are both the identity. Hence $\Phi:\cM\to\cN\bar{\otimes} L^\infty(C,\nu_C)$ is an injective conditional expectation. Therefore, we must have $\Phi=\id_{\cN\bar{\otimes} L^\infty(C)}$ and $\cM=\cN\bar{\otimes} L^\infty(C,\nu_C)$, which finishes the proof.
\end{proof}

\section{Ideals in Tensor product of von Neumann algebras}
\label{sec:obstruction}
Up to now, we have established \say{product–rigidity} results under
ergodicity on one leg. In this section, we investigate the
opposite end of the spectrum. In particular,  we show that the presence of a non–trivial (WOT–closed) ideal in the
tensor product may produce canonical intermediate subalgebras that fail to split.
\begin{lemma}
\label{closedinv}
Let $N\subseteq M$ be von Neumann algebras. Assume that $I\lhd M$ is a wot-closed (two-sided) ideal in $M$. Then $P:=N+I$ is a von Neumann subalgebra in $M$.
\end{lemma} 
\begin{proof}
We just need to show $P$ is closed under wot.

Let $x_n+y_n\in P$ be a net converging to $z\in M$ under wot, where $x_n\in N$ and $y_n\in I$ for all $n\geq 1$. We aim to show $z\in N+I$.

Recall that since $I$ is an ideal closed under wot, we may find some central projection $p\in M$ such that $I=Mp$. Below, when we write $\lim$, it should be understood as limit under wot.

Now, $z(1-p)=\lim_n(x_n+y_n)(1-p)=\lim_nx_n(1-p)\in N(1-p)$ since $x_n\in N$ and $N(1-p)$ is a von Neumann subalgebra in $Mp$ and hence closed under wot.
So we may write $z(1-p)=\lim_nx_n(1-p)=n(1-p)$ for some $n\in N$.

Therefore, we have 
\begin{align*}
    z=z(1-p)+zp=n(1-p)+zp=n+(z-n)p\in N+Mp=N+I.
\end{align*}
\end{proof}
Let $(\mathcal{M},\tau_{\mathcal{M}})$ and $(\mathcal{N},\tau_{\mathcal{N}})$ be trace-preserving $G$-von Neumann algebras. Let $\mathcal{P}=\mathcal{M}\bar{\otimes}\mathcal{N}$, and $\tau=\tau_{\mathcal{M}}\otimes\tau_{\mathcal{N}}$. In this case, we have a $G$-equivariant $\tau$-preserving canonical conditional expectations $\mathbb{E}_{\mathcal{M}}:\mathcal{P}\to\mathcal{M}$ defined by $\mathbb{E}_{\mathcal{M}}(x)=(\text{id}\otimes\tau_N)(x)$. Similarly, $\mathbb{E}_{\mathcal{N}}(x)=(\tau_M\otimes\text{id})(x)$ is the canonical $\tau$-preserving conditional expectation onto $\mathcal{N}$.
\begin{proposition}
\label{objectassociatedwithanideal}
Let $(\mathcal{M},\tau_{\mathcal{M}})$ and $(\mathcal{N},\tau_{\mathcal{N}})$ be trace-preserving $G$-von Neumann algebras. Let $\mathcal{P}=\mathcal{M}\bar{\otimes}\mathcal{N}$. Assume that $\mathcal{M}$ and $\mathcal{N}$ do not admit non-trivial $G$-invariant WOT-closed two-sided ideals. Let $I\triangleleft \mathcal{P}$ be a $G$-invariant two-sided closed non-trivial ideal, hence, is of the form $q\mathcal{P}$ for some $G$-invariant projection $q\in\mathcal{Z}(\mathcal{P})$. Let\[\mathcal{V}_{I,~\mathcal{N}}=\mathcal{N}+q\mathcal{P}\text{ and }\mathcal{V}_{I,~\mathcal{M}}=\mathcal{M}+q\mathcal{P}\]
Then the following hold true:
\begin{enumerate}
    \item\label{intmobject} $\mathcal{V}_{I,~\mathcal{N}}$ and  $\mathcal{V}_{I,~\mathcal{M}}$ are $G$-invariant von Neumann subalgebras of $\mathcal{P}$ containing $\mathcal{N}$ and $\mathcal{M}$ respectively.
    \item\label{spliteverything} If $\mathcal{V}_{I,~\mathcal{N}}$ splits, then $\mathcal{V}_{I,~\mathcal{N}}=\mathcal{P}$. Similarly, if $\mathcal{V}_{I,~\mathcal{M}}$ splits, it must be equal to $\mathcal{P}$. 
\end{enumerate}
Moreover, if $\mathcal{V}_{I,~\mathcal{N}}$ splits, then $p=1-q$ is separating for both $\mathcal{M}$ and $\mathcal{N}$ if we further assume that the actions $G\curvearrowright \mathcal{M}$ and $G\curvearrowright\mathcal{N}$ are ergodic. Moreover, in this case, there is a trace-preserving injective $*$-homomorphism $\phi: \mathcal{M}\to\mathcal{N}$.
\begin{proof}
$(\ref{intmobject})$ follows from Proposition~\ref{closedinv}. We also observe that $\overline{\mathbb{E}_{\mathcal{M}}(q\mathcal{P})}^{\text{WOT}}$ is a two-sided $G$-invariant WOT-closed in $\mathcal{M}$. Since $\mathcal{M}$ does not have any non-trivial WOT-closed two-sided ideals, it must be the case that  $\overline{\mathbb{E}_{\mathcal{M}}(q\mathcal{P})}^{\text{WOT}}$ is either $\mathcal{M}$ or $0$. Since $\mathbb{E}_{\mathcal{M}}$ is faithful, we obtain that $\overline{\mathbb{E}_{\mathcal{M}}(q\mathcal{P})}^{\text{WOT}}=\mathcal{M}$. A similar argument shows that $\overline{\mathbb{E}_{\mathcal{N}}(q\mathcal{P})}^{\text{WOT}}=\mathcal{N}$. This shows that $$1\otimes\mathcal{N}\subsetneq\mathcal{V}_{I,~\mathcal{N}}\subseteq\mathcal{P}.$$ Similarly, we can conclude that $$\mathcal{M}\otimes 1\subsetneq\mathcal{V}_{I,~\mathcal{M}}\subseteq\mathcal{P}.$$
Let us now assume that $\mathcal{V}_{I,~\mathcal{N}}$ splits, i.e., $\mathcal{V}_{I,~\mathcal{N}}=\mathcal{M}_0\otimes\mathcal{N}$ for some $G$-invariant von Neumann subalgebra $\mathcal{M}_0\subset\mathcal{M}$. Then, $\mathbb{E}_{\mathcal{M}}(\mathcal{V}_{I,~\mathcal{N}})\subset\mathcal{V}_{I,~\mathcal{N}}$. Consequently, we see that
\[\mathcal{M}=\overline{\mathbb{E}_{\mathcal{M}}(q\mathcal{P})}^{\text{WOT}}\subset\overline{\mathbb{E}_{\mathcal{M}}(\mathcal{V}_{I,~\mathcal{N}})}^{\text{WOT}}\subset\mathcal{V}_{I,~\mathcal{N}}\]
This shows that $\mathcal{V}_{I,~\mathcal{N}}=\mathcal{P}$. A similar argument shows that $\mathcal{V}_{I,~\mathcal{M}}=\mathcal{P}$ under the assumption that $\mathcal{V}_{I,~\mathcal{M}}$ splits. 

Let us now assume that $\mathcal{V}_{I,\mathcal{N}}$ splits. Then, $\mathcal{N}+q\mathcal{P}=\mathcal{P}$. Multiplying by $1-q$, we obtain that $(1-q)\mathcal{N}=(1-q)\mathcal{P}$. Since $(1-q)\mathcal{N}=(1-q)\mathcal{P}$, we see that $(1-q)\mathcal{M}\subset(1-q)\mathcal{N}$. Let $p=(1-q)$. We claim that $p$ is separating for both $\mathcal{M}$ and $\mathcal{N}$. Indeed, if $px=0$ for some $x\in\mathcal{M}$, applying the canonical conditional expectation $\mathbb{E}_{\mathcal{M}}$ on both sides, we see that $\mathbb{E}_{\mathcal{M}}(p)x=0$. Since $\mathbb{E}_{\mathcal{M}}(p)$ is a $G$-invariant element in $\mathcal{M}$ and $G\curvearrowright (\mathcal{M},\tau_{\mathcal{M}})$ is ergodic, it must be the case that $\mathbb{E}_{\mathcal{M}}(p)\in\mathbb{C}$. However if $\mathbb{E}_{\mathcal{M}}(p)=0$, then we would have that $p=0$ from whence we would have that $q=1$ which contradicts the assumption that $I$ is non-trivial. Therefore, $\mathbb{E}_{\mathcal{M}}(p)\in\mathbb{C}\setminus\{0\}$. Consequently, it follows that $x=0$. A similar argument shows that if $py=0$ for some $y\in\mathcal{N}$, then $y=0$.

We can now define a $G$-equivariant $*$-homomorphism $\phi:\mathcal{M}\to\mathcal{N}$, where $\phi(x)$ is determined by $xp=\phi(x)p$ using the above proved fact that $p\mathcal{M}\subset p\mathcal{N}$. Note that this is a $*$-homomorphism since for $x,y\in\mathcal{M}$,
\[\phi(xy)p=xyp=(xp)(yp)=\phi(x)p\phi(y)p=\phi(x)\phi(y)p\]
Here we have used that $p$ is separating for $\mathcal{N}$. Moreover, $\phi$ is injective since $p$ is separating for $\mathcal{M}$. We now show that $\phi$ is trace-preserving in the sense that $\tau_{\mathcal{N}}(\phi(x))=\tau_{\mathcal{M}}(x)$ for all $x\in\mathcal{M}$.
Since both $\mathbb{E}_{\mathcal{M}}(p)$ and $\mathbb{E}_{\mathcal{N}}(p)$ are constants. Moreover,
\[\mathbb{E}_{\mathcal{M}}(p)=\tau\left(\mathbb{E}_{\mathcal{M}}(p)\right)=\tau(p)=\tau\left(\mathbb{E}_{\mathcal{N}}(p)\right)=\mathbb{E}_{\mathcal{N}}(p)\]
Now, for an element $m\in\mathcal{M}$, we see that
\[\tau(mp)=\tau\left(\mathbb{E}_{\mathcal{M}}(mp)\right)=\mathbb{E}_{\mathcal{M}}(p)\tau(m)=\tau(p)\tau(m).\]
Similarly,
\[\tau(\phi(m)p)=\tau\left(\mathbb{E}_{\mathcal{N}}(\phi(m)p)\right)=\mathbb{E}_{\mathcal{N}}(p)\tau(\phi(m))=\tau(p)\tau(\phi(m)).\]
Since $mp=\phi(m)p$, we see that $\tau(p)\tau(m)=\tau(p)\tau(\phi(m))$ for all $m\in\mathcal{M}$. Since $\tau(p)\ne 0$, it follows that $\tau(m)=\tau(\phi(m))$ for all $m\in\mathcal{M}$. This in turn implies that for all $m\in\mathcal{M}$,
\[\tau_{\mathcal{M}}(m)=\tau\left(\mathbb{E}_{\mathcal{M}}(m)\right)=\tau(m)=\tau(\phi(m))=\tau\left(\mathbb{E}_{\mathcal{N}}(\phi(m))\right)=\tau_{\mathcal{N}}(\phi(m)).\]
\end{proof}
\end{proposition}
Consequently, if one leg is abelian and the other is genuinely
non–commutative, such a $\phi$ cannot exist; hence the corresponding intermediate does
not split. We make this precise below.
\begin{thm}
\label{thm:gensplitidealvna}
Let $(\mathcal{M},\tau_{\mathcal{M}})$ and $(\mathcal{N},\tau_{\mathcal{N}})$ be trace-preserving $G$-von Neumann algebras. Let $\mathcal{P}=\mathcal{M}\bar{\otimes}\mathcal{N}$. Assume that $G\curvearrowright(\mathcal{M},\tau_{\mathcal{M}})$ and $G\curvearrowright(\mathcal{N},\tau_{\mathcal{N}})$ are ergodic. Let $I\triangleleft \mathcal{P}$ be a $G$-invariant two-sided closed non-trivial ideal, hence, is of the form $q\mathcal{P}$ for some $G$-invariant projection $q\in\mathcal{Z}(\mathcal{P})$. Let\[\mathcal{V}_{I,~\mathcal{M}}=\mathcal{M}+q\mathcal{P}\]    
Moreover, assume that $\left(\mathcal{M},\tau_{\mathcal{M}}\right)$ is abelian and is of the form $L^{\infty}(X,\nu)$ and $(\mathcal{N},\tau_{\mathcal{N}})$ is non-commutative. Then, $\mathcal{V}_{I,~\mathcal{M}}$ does not split. 
\begin{proof}
Note that the ergodicity assumption implies that  $\mathcal{M}$ and $\mathcal{N}$ do not admit non-trivial $G$-invariant WOT-closed two-sided ideals. 
Let us assume that $\mathcal{V}_{I,~\mathcal{M}}$ splits. 
It follows from Proposition~\ref{objectassociatedwithanideal} that there is a trace preserving injective $*$-homorphism $\phi:(\mathcal{N},\tau_{\mathcal{N}})\to  L^{\infty}(X,\nu) $. This is a contradiction to the assumption that $(\mathcal{N},\tau_{\mathcal{N}})$ is non-commutative. 
\end{proof}
\end{thm}
Below, we prove the measurable version of the above proposition.

\begin{proposition}
\label{prop:abelianvnanonsplitting}
Let $\Gamma\curvearrowright (X,\nu)$ and $\Gamma\curvearrowright (Y,\eta)$ be two  ergodic pmp actions on atomless standard probability spaces. Assume that the diagonal action $\Gamma\curvearrowright (X\times Y,\nu\times \eta)$ is not ergodic, say $Z\subseteq X\times Y$ is a $G$-invariant Borel subset with $0<(\nu\times\eta)(Z)<1$. Set $q=\chi_Z\in L^{\infty}(X\times Y,\nu\times\eta)$ be the characteristic function on $Z$.  Set $\mathcal{Q}=L^{\infty}(Y,\eta)+L^{\infty}(X\times Y,\nu\times\eta)q$. Then $\mathcal{Q}$ is a $\Gamma$-invariant von Neumann subalgebra between $L^{\infty}(Y,\eta)$ and $L^{\infty}(X\times Y,\nu\times\eta)$ which does not split as $L^{\infty}(X'\times Y,\nu'\times\eta)$ for any $G$-factor map $\pi: (X,\nu)\rightarrow (X',\nu')$.  
\end{proposition}
\begin{proof}
Assume that $\mathcal{Q}$ splits.
Using Theorem~\ref{thm:gensplitidealvna}, we get a measure-preserving $\Gamma$-equivariant injective *-homomorphism $\phi: L^{\infty}(X,\nu)\rightarrow L^{\infty}(Y,\eta)$. Hence it corresponds to a $\Gamma$-factor map $\pi: (Y,\eta)\rightarrow (X,\nu)$, i.e., $\phi(a)(y)=a(\pi(y))$.

Finally, let us check that $Z^c\subseteq \{(\pi(y),y):~y\in Y\}$.

Indeed, note that now the identity $ap=\phi(a)p$ as in the proof of Proposition \ref{objectassociatedwithanideal} reads as \[a(x)\chi_{Z^c}(x, y)=\phi(a)(y)\chi_{Z^c}(x, y)=a(\pi(y))\chi_{Z^c}(x, y)\] for a.e. $(x, y)\in X\times Y$ and all $a\in L^{\infty}(X)$. 

Hence, for a.e. $(x, y)\in Z^c$,  we get $a(x)=a(\pi(y))$ for all $a\in L^{\infty}(X)$, therefore, $x=\pi(y)$.

From what we have proved, we deduce that 
$c=(\nu\times \eta)(Z^c)\leq (\nu\times \eta)(\{(\pi(y), y): y\in Y\})=0$ since $\nu$ is atomless, a contradiction.
\end{proof}
\section{On the \texorpdfstring{$SL_2(\mathbb{Z})$}{}-actions}
\label{sec:SL_2(Z)}
In this section, we focus on the  \say{$\infty$–entropy} Bernoulli leg
\(G\curvearrowright (Y,\nu)=(\mathbb T^{G},\mu_0^{G})\)
and the structured toral action
\(G\curvearrowright (X,\mu)=(\mathbb T^2,\mu)\) for \(G=SL_2(\mathbb Z)\),
and we classify all intermediate factors between
\(G\curvearrowright (Y,\nu)\times (X,\mu)\) and \(G\curvearrowright (Y,\nu)\).
In particular, we show that every such intermediate is forced to be a
product \( (Y,\nu)\times (Z,\zeta) \), with \( (Z,\zeta) \) a factor of \( (X,\mu) \).

\begin{thm}
\label{thm:toricaction}
Let $G=SL_2(\mathbb{Z})\curvearrowright (Y,\nu)=(\mathbb{T}^{G},\mu_0^G)$ be the Bernoulli shift and let $G\curvearrowright (X,\mu)=(\mathbb{T}^2, \mu)$ be the standard action with $\mu$ being the Haar measure. Let $G\curvearrowright (Q,\eta)$ be an intermediate factor between $G\curvearrowright (Y,\nu)\times (X,\mu)$ and $G\curvearrowright (Y,\nu)$. Then there exists a factor map $G\curvearrowright (X,\mu)\rightarrow (Z,\zeta)$ such that $G\curvearrowright (Q,\eta)$ is measurably conjugate to $G\curvearrowright (Y,\nu)\times (Z,\zeta)$. 
\end{thm}
\begin{proof}
The proof consists of two steps.

Step 1. we show the conclusion holds but with $G\curvearrowright (Z,\zeta)$ only known to be a quasi-factor of $G\curvearrowright (X,\mu)$ for now.

To show this, it suffices to check the assumptions in \cite[Theorem 1.1]{glasner2023intermediate} holds in order to apply this theorem to our setting. In other words, we need to show $G\curvearrowright(Y,\nu)$ is disjoint from any ergodic (joining) quasi-factor of $G\curvearrowright (X,\mu)$.

Let $G\curvearrowright (Z,\zeta)$ be any such a quasi-factor. Let $\theta$ be a joining for $G\curvearrowright (Y,\nu)$ and $G\curvearrowright (Z,\zeta)$. Note that it is still a joining for the two subactions $H\curvearrowright (Y,\nu)$ and $H\curvearrowright (Z,\zeta)$, where $H=\begin{pmatrix}
1&\mathbb{Z}\\
0&1
\end{pmatrix}\leq G$ and $H\curvearrowright (Z,\zeta)$ is still a quasi-factor of $H\curvearrowright (X,\mu)$.

Then, we note that $H\curvearrowright (X,\mu)$ has zero entropy by \cite[Theorem 18.19]{Glasner-book} since it is distal and hence as a quasi-factor, $H\curvearrowright (Z,\zeta)$ also has zero entropy by \cite[Theorem 18.17]{Glasner-book}. Besides, $H\curvearrowright (Y,\nu)$ is still a Bernoulli shift and hence has completely positive entropy by \cite[Proposition 3.51 and Lemma 18.7(3)]{Glasner-book}, see the comment after \cite[Definition 18.8]{Glasner-book}. In fact it is known that Bernoulli shifts of any sofic groups have completely positive entropy by Kerr's result \cite{kerr}. Therefore, the two actions $H\curvearrowright (Y,\nu)$ and $H\curvearrowright (Z,\zeta)$ are disjoint by \cite{furstenberg}, from which we deduce that $\theta=\nu\times \zeta$, thus $G\curvearrowright (Y,\nu)$ and $G\curvearrowright (Z,\zeta)$ are disjoint.

Step 2. we show that the above quasi-factor $G\curvearrowright (Z,\zeta)$ is actually a factor of $G\curvearrowright (X,\mu)$.

Note that we have $G\curvearrowright (Z,\zeta)$ is a factor of $G\curvearrowright (Y,\nu)\times (X,\mu)$. Indeed, this is because we have the following successive factor maps:
\begin{align*}
    G\curvearrowright (Y,\nu)\times (X,\mu)\rightarrow (Q,\eta)\cong (Y,\nu)\times (Z,\zeta)\rightarrow (Z,\zeta).
\end{align*}
Once again, we note that by considering the $H$-subactions, we have
$H\curvearrowright (X,\mu)$ is the Pinsker factor of $H\curvearrowright (Y,\nu)\times (X,\mu)$. This relies on the following three facts:
\begin{itemize}
\item $H\curvearrowright (Y,\nu)$ is the Bernoulli shift and hence has completely positive entropy and therefore has trivial Pinsker factor.
\item $H\curvearrowright (X,\mu)$ has zero entropy and hence coincides with its Pinsker factor.
\item The Pinsker factor of a diagonal action is equal to the diagonal action on the product of two Pinsker factors \cite{gtw}. See also \cites{dan,ben-pinsker}.
\end{itemize}
Moreover, notice that $H\curvearrowright(Z,\zeta)$ has zero entropy as shown before, we therefore deduce that as a factor of the subaction $H\curvearrowright (Y,\nu)\times (X,\mu)$, we have that
$H\curvearrowright (Z,\zeta)$ is actually a factor of $H\curvearrowright (X,\mu)$. This means we can identify $L^{\infty}(Z,\zeta)$ as a von Neumann subalgebra of $L^{\infty}(X,\mu)$. Note that $L^{\infty}(Z,\zeta)$ is already $G$-invariant as we start with a $G$-action $G\curvearrowright (Z,\zeta)$. Hence, we deduce that $L^{\infty}(Z,\zeta)$ is a $G$-invariant von Neumann subalgebra in $L^{\infty}(X,\mu)$, which means we have a $G$-factor map $G\curvearrowright (X,\mu)\rightarrow (Z,\zeta)$. This finishes the proof.
\end{proof}
Initiated in \cite{haagerup1978example}, the Haagerup property has been highly fruitful for operator algebras-offering geometric insight, a weakening of amenability that retains workable techniques, and a robust antithesis to Property~(T). A systematic study of the notion of maximal Haagerup subalgebras was initiated in \cite{jiangskalski} and continued further in \cite{jiang2021maximalsolo}. Using Theorem~\ref{thm:toricaction}, we show that $L^{\infty}(Y,\nu)\rtimes G$ is a maximal Haagerup von Neumann subalgebra in $L^{\infty}((Y,\nu)\times (X,\mu))\rtimes G$, where $X$, $Y$, and $G$ are as in the theorem. In particular, this answers the second part of \cite[Problem 5.2]{jiangskalski}.
\begin{cor}
Let $G=SL_2(\mathbb{Z})\curvearrowright (Y,\nu)=(\mathbb{T}^{G},\mu_0^G)$ be the Bernoulli shift and let $G\curvearrowright (X,\mu)=(\mathbb{T}^2, \mu)$ be the standard action with $\mu$ being the Haar measure. Then $L^{\infty}(Y,\nu)\rtimes G$ is a maximal Haagerup von Neumann subalgebra in $L^{\infty}((Y,\nu)\times (X,\mu))\rtimes G$.
\end{cor}
\begin{proof}
Let $L^{\infty}(Y,\nu)\rtimes G\subsetneq N\subseteq L^{\infty}((Y,\nu)\times (X,\mu))\rtimes G$ be an intermediate von Neumann subalgebra. Then since $G\curvearrowright (Y,\nu)$ is an essential free action, we deduce from Suzuki's theorem \cite{suzuki} that $N=L^{\infty}(Q,\eta)\rtimes G$ for some intermediate factor $G\curvearrowright (Q,\eta)$ between $G\curvearrowright (Y,\nu)$ and $G\curvearrowright (Y,\nu)\times (X,\mu)$. By Theorem~\ref{thm:toricaction}, we deduce that $G\curvearrowright (Q,\eta)$ is conjugate to $G\curvearrowright (Y,\nu)\times (Z,\zeta)$ for some $G$-factor map $G\curvearrowright (X,\nu)\rightarrow (Z,\zeta)$. As $N\neq L^{\infty}(Y,\nu)\rtimes G$, we get that $G\curvearrowright (Z,\zeta)$ is not a trivial action and hence by \cite[lemma 3.5]{jiangskalski}, we know that $N$ has a non-Haagerup von Neumann subalgebra $L^{\infty}(Z,\zeta)\rtimes G$ and hence $N$ does not have the Haagerup property itself.  
\end{proof}
\newpage
\bibliographystyle{}\bibliography{name}
\end{document}